\documentclass[10pt]{amsart}
\usepackage{amsmath}
\usepackage{amsfonts}
\usepackage{amssymb}
\usepackage{mathrsfs}
\usepackage[all]{xy}
\usepackage{url}

\newcommand{\Vsp}{{\boldsymbol V}}
\newcommand{\sett}[1]{\ensuremath{\left \{ #1 \right \}}}
\newcommand{\ip}[2]{\ensuremath{\left<#1,#2\right>}}
\newcommand{\abs}[1]{\ensuremath \left| #1 \right|}
\newcommand{\el}{z}
\newcommand{\rel}{\textit{rel}}

\newcommand{\schurw}{{\mathcal{S}_w}}
\newcommand{\Sobqs}{{L^q_s}}
\newcommand{\Bsploc}{\Bsp_{\textit{loc}}}
\newcommand{\spline}{\Vsp}
\newcommand{\splinepv}{\spline^p_v}
\newcommand{\splinepvi}{\left(\spline^p_v\right)_i}
\newcommand{\splinet}{\spline^2}

\newcommand{\splinezz}{\spline^{00}}

\newcommand{\stft}{\mathcal{V}}
\newcommand{\parti}{{\ensuremath{E_i}}}
\newcommand{\partj}{{\ensuremath{E_j}}}
\newcommand{\partt}{\ensuremath{\mathcal{E}}}
\newcommand{\partn}{\#_{\mathcal{E}}}
\newcommand{\partnq}{\#_{\mathcal{E,Q}}}
\newcommand{\partnqp}{\#_{\mathcal{E,Q'}}}
\newcommand{\partfi}{\eta_i}
\newcommand{\Linfloc}{L^\infty_{\textit{loc}}}

\newcommand{\gabt}{\mathbf{G}}
\newcommand{\dash}{-}
\newcommand{\field}[1]{\mathbb{#1}}
\newcommand{\bC}{\field{C}}
\newcommand{\Rst}{{\mathbb R}}
\newcommand{\RR}{{\mathbb R}}
\newcommand{\Rdst}{{\Rst^d}}
\newcommand{\RRd}{{\Rst^d}}
\newcommand{\Rd}{{\Rst^d}}
\newcommand{\Rtdst}{{\Rst^{2d}}}
\newcommand{\set}[2]{\big\{ \, #1 \, \big| \, #2 \, \big\}}
\newcommand{\norm}[1]{\lVert#1\rVert}
\newcommand{\Esp}{{\boldsymbol E}}
\newcommand{\Bsp}{{\boldsymbol B}}
\newcommand{\Hsp}{{\boldsymbol H}}
\newcommand{\DRd}{{{\mathcal D}(\Rst^d)}}
\newcommand{\SchRd}{\mathcal{S}(\RRd)}
\newcommand{\SchpRd}{{\mathcal{S}'(\Rdst)}}
\newcommand{\SchpRtd}{\mathcal{S}'(\Rtdst)}
\newcommand{\ebbes}{\mbox{$\,\cdot\,$}}
\newcommand{\BspN}{(\Bsp, \, \|\ebbes\|_\Bsp)}

\newcommand{\Zst}{{\mathbb Z}}
\newcommand{\Zdst}{{\Zst^d}}
\newcommand{\supess}{\mathop{\operatorname{supess}}}
\newcommand{\supp}{\operatorname{supp}}
\newcommand{\Lsp}{{\boldsymbol L}}
\newcommand{\Ltsp}{{\Lsp^2}}
\newcommand{\LtRd}{{\Ltsp(\Rst^d)}}
\newcommand{\Nst}{{\mathbb N}}

\newtheorem{lemma}{Lemma}
\newtheorem{theorem}{Theorem}
\newtheorem{corollary}{Corollary}
\newtheorem{prop}{Proposition}

\newtheorem{rem}{Remark}

\newtheorem{example}{Example}

\title{Surgery of spline-type and molecular frames}

\author[J.L.~Romero]{Jos\'e Luis Romero}
\address{Departamento de
Matem\'atica \\ Facultad de Ciencias Exactas y Naturales\\ Universidad
de Buenos Aires\\ Ciudad Universitaria, Pabell\'on I\\ 1428 Capital
Federal\\ ARGENTINA\\ and CONICET, Argentina}
\email[Jos\'e Luis Romero]{jlromero@dm.uba.ar}

\begin{document}
\begin{abstract}
We prove a result about producing new frames for general spline-type spaces
by piecing together portions of known frames. Using spline-type spaces as
models for the range of certain integral transforms, we obtain results
for time-frequency decompositions and sampling.
\end{abstract}

\maketitle

\section{Introduction and overview}
By a spline-type space we mean a normed function space generated by well
localized atoms. The most commonly found scenario in the literature is the
one of finitely generated shift-invariant spaces. In that case, the atoms are
simply lattice shifts of a finite family of functions. Our motivation to
consider spline-type spaces in more generality comes from its possible
applications to atomic decompositions. Many atomic decompositions of
classical function spaces are produced by sampling a continuous integral
(wavelet) transform. The form of that integral transform makes it possible to
relate its size and spatial localization to its modulus of continuity. This
observation plays a central role, for example, in the general theory
developed in \cite{fegr89}. In this article we argue that spline-type spaces
are a model for the range of certain wavelet transforms and prove a locality
statement for them that can then be recast as a result for the
corresponding atomic decompositions.

We will prove a locality principle in the form of a \emph{surgery scheme}
for well-localized frames. Our main result asserts that, given a family of
frames for a spline-type space, it is possible to construct a new frame for
the same space by piecing together arbitrary portions of the original frames,
provided that the overlaps between these portions are large enough. Although
the result we prove is qualitative, special emphasis is made on how the
qualities of the ingredients affect the surgery procedure and what kind of
uniformity is to be expected. This is one reason why we work on the Euclidean
space and not on a general locally-compact group (although much of the
elements involved in our construction have a counterpart in the abstract
setting.) The other reason is that we make use of localization theory (see
\cite{gr04-1}, \cite{fogr05} and \cite{bacahela06}) and the most
interesting examples we could cover by extending the setting to abstract
groups, are not covered by it.

For the applications we consider mainly two transforms. The first is the
Short Time Fourier Transform (STFT) with a fixed (good) window. This
transform maps modulation spaces into spline-type spaces - considered in the
general sense - and then yields an application of the surgery scheme to
Gabor frames. These results imply a general existence condition for the
recently introduced concept of \emph{quilted Gabor frame} (see
\cite{dofe09}.) Since the STFT does not exactly map time-frequency shifts
into translations - there is an extra phase factor or twist on the STFT side
- we see that shift-invariant spaces are not a sufficient model for the range
of the transform: we must us general spline-type spaces. As a by-product of
this general treatment, the result we get holds not only for pure
time-frequency shifts but also for Gabor molecules concentrated around a
general set of nodes.

The second transform we consider is the Kohn-Nirenberg map, which
- as shown in \cite{fe02} -
establishes a correspondence between the class of Gabor multipliers 
(related to different Gabor frames) and
the class of (shift-invariant) spline-type spaces (see also
\cite[Chapter 5]{fest03}.) Gabor multipliers are operators that arise from
applying a mask to the coefficients associated with a Gabor frame expansion;
hence each of these operators has the form
\[
T= \sum_{\lambda \in \Lambda} c_\lambda P_\lambda,
\]
where $c_\lambda \in \bC$ and $P_\lambda$ is a rank-one operator
(essentially a projector onto the subspace generated by a time-frequency
atom.) Each operator in a given class of Gabor multipliers can be identified
by its associated \emph{lower symbol} which consists of the Hilbert-Schmidt
inner products $\set{\ip{T}{P_\lambda}}{\lambda \in \Lambda}$. Combining the
surgery scheme with the KN map and known tools for shift-invariant spaces we
get a sufficient condition to identify a class of Gabor multipliers by a
\emph{mixed lower symbol} constructed by using different types of rank-one
operators $P_\lambda$ for $\lambda$ in different regions of the
time-frequency plane.

Finally, we give an application to irregular sampling. Given a family of
sampling sets for which a sampling inequality is known, we can construct new
sets for which the sampling inequality still holds. Moreover, given explicit
reconstruction formulas for the original sets, we get an approximate
reconstruction formula for the new sets.

We now mention two technical aspects of the article. The first one concerns
the use of localization theory. In order to develop the surgery scheme, we
not
only need to know that localized frames have localized dual frames but also
what qualities of the original atoms influence the concentration of their
dual atoms. To this end, we resort to the constructive proof of Wiener's
lemma for infinite matrices given in \cite{su07-5} (see also \cite{su05}.)

The second aspect concerns the use of amalgam spaces. When moving to the
setting of spaces generated by general atoms, the standard tools for amalgam
spaces are not directly applicable and require an extension. In the study
of shift-invariant spaces (or more generally, spaces generated by translates)
the relevant operators can be expressed as products and convolutions with
possibly distributional kernels. Wiener amalgam spaces, as introduced in
\cite{fe83}, have proved to be a powerful tool to quantify this formalism.
The abstract convolution multiplier theorems allow to deal with smoothness
and approximation problems in the context of atoms generated by irregular
shifts (see \cite{fe92-3}.)
In the context of general spline-type spaces, the relevant operations are not convolutions but, nonetheless, they are convolution-like. For example, in the proposed applications to (regular) Gabor frames, instead of convolution inequalities for Wiener amalgams we would need \emph{twisted} convolution inequalities.

Convolution dominated operators (see \cite{ba90-1}, \cite{ja90-2},
\cite{fegrle08}) and enveloping conditions for irregular atoms (\cite{su08},
\cite{bacahela06}, \cite{grpi09}) are now widely used concepts. Here, we will
consider an enveloping condition for atoms, not in a pointwise sense, but in
the sense of a local - possibly non solid - quantity. We will extend the
amalgam norm of a function $f$ to families of functions $F$ in such a way
that the condition $\norm{F}_{W(B,E)} < \infty$ grants to $F$ the same
properties shared by a set of translates of $f$, when $\norm{f}_{W(B,E)} <
\infty$. When the local norm measures size, the condition $\norm{F}_{W(B,E)}
< \infty$ will amount to certain spatial localization for the family
$F$; when the local norm measures smoothness, it will amount to certain
equismoothness property for the family $F$. Using this extension of the
amalgam norm and a simple interpolation argument, we obtain replacements
for some of the convolution inequalities in amalgam spaces. These are needed,
for example, to extend to the general setting the principle that in a
finitely-generated shift-invariant space the smoothness of the generating
windows is inherited by the whole space. 

The article is organized as follows. Section \ref{sec_prelim} introduces all
the tools required to deal with general sets of atoms and in particular the
extension of the amalgam norms to families of functions. In Section
\ref{sec_splines} we formally introduce spline-type spaces and extend to this
setting some of the classic results for shift-invariant spaces. In Section
\ref{sec_surgery} we prove the central result on piecing together various
frames to build a new one. The result stated there has the limitation of
requiring specific information on the decay of the dual atoms of the frames
being pieced together. For clarity, this limitation is addressed separately
in Section \ref{sec_quality}. Section \ref{sec_applications} collects the
results from the previous sections and gives several applications.
\section{Preliminaries}
\label{sec_prelim}
\subsection{Frames}
\label{sec_frames}
Let $\Esp$ be a separable Banach space. A \emph{Banach frame} for $\Esp$
consists of a countable family $\sett{f_k}_{k\in\Lambda} \subseteq \Esp$
together with a sequence space $\Esp_d \hookrightarrow \bC^\Lambda$ such that
the \emph{reconstruction operator}
\begin{align}
\label{eq_R_banach_frame}
R: \Esp_d &\to \Esp
\\
\nonumber
\left(c_k\right)_k &\mapsto \sum_k c_k f_k
\end{align}
is a well-defined bounded \emph{retraction}. This means that $R$ is bounded
and there exists another operator $C: \Esp \to \Esp_d$, called
\emph{coefficient operator}, such that $RC=I_\Esp$. We will only consider
sequence spaces $\Esp_d$ for which the set of sequences $\sett{\delta^k}_k$
- given by $\delta^k_j= 1$, if $k=j$, and $0$ otherwise - is an
unconditional basis. In this case, if we only assume that the operator $R$
maps each sequence $\delta^k$ to $f_k$, then it follows that $R$ has the
form prescribed by Equation \eqref{eq_R_banach_frame} and that the series
in that equation converges unconditionally.

Since $\Esp_d \hookrightarrow \bC^\Lambda$, the coefficient operator $C$ is
implemented by some family of linear functionals $\sett{g_k}_k \subseteq
\Esp'$ by means of the formula $C(f) = \left( \ip{f}{g_k} \right)_k$. When a
particular choice of a coefficient operator (and hence of coefficient
functionals $\sett{g_k}_k$) is made, we speak of a \emph{frame pair}
$\left(\sett{f_k}_k, \sett{g_k}_k \right)$.

In concrete examples, the coefficient functionals may have various representations. For example, if $\Esp$ is a closed subspace of a Hilbert space $\Hsp$, then each coefficient
functional can be represented by various vectors $g_k \in \Hsp$. Each of
these choices will be considered to yield a different frame pair
$\left(\sett{f_k}_k, \sett{g_k}_k \right)$.

It is common in the literature to define Banach frames in terms of
the coefficients functionals $\sett{g_k}_k$ rather than the atoms
$\sett{f_k}_k$. A family of linear functionals $\sett{g_k}_{k\in\Lambda}
\subseteq \Esp'$, together with a sequence space
$\Esp_d \hookrightarrow \bC^\Lambda$  is called a \emph{Banach frame} for $\Esp$ 
if the coefficient operator
\begin{align*}
C: \Esp &\to \Esp_d
\\
f &\mapsto \left( \ip{f}{g_k} \right)_k
\end{align*}
is a bounded \emph{section}; that is, there exists a bounded operator $R: \Esp_d \to \Esp$ such that
$RC=I_\Esp$.
In the abstract setting there is no possible confusion between the two usages
since the atoms and the coefficients functionals belong to
different spaces. However, in concrete examples where $\Esp$ is a classical
function space and $\Esp'$ is identified with another classical function
space, these two usages can be ambiguous. In the context of Hilbert spaces,
since $C=R^*$, if we use the canonical representation of functionals as
elements of the Hilbert space, then the two definitions of frames are
equivalent. This equivalence extends to \emph{localized frames} (see
\cite{gr04-1}.)
Frames produced by extension of a localized Hilbert space frame to its
associated Banach spaces are Banach frames in both of the senses discussed
above (see \cite{fogr05}.) In this article, every reference to a Banach frame
will be followed by a clarification about its precise meaning.

If $\Hsp$ is a Hilbert space and $\Esp \subseteq \Hsp$ is a closed subspace,
a sequence $F \equiv \sett{f_k}_k \subseteq \Hsp$ whose orthogonal projection
onto $\Esp$ forms a frame for $\Esp$ is called an \emph{exterior frame} for
$\Esp$. Hence, $F$ is an exterior frame for $\Esp$ if it satisfies the
\emph{exterior frame inequality},
\begin{equation*}
 A \norm{g}^2_\Hsp \leq \sum_k \abs{\ip{g}{f_k}}^2 \leq B \norm{g}^2_\Hsp,
\quad \mbox{for all $g \in \Esp$,}
\end{equation*}
and some constants $0<A \leq B < +\infty$.
Likewise, if the projection of $F$ onto $\Esp$, forms a Riesz basis for
$\Esp$, then $F$ is called a \emph{Riesz projection basis} for $\Esp$ (see
for example \cite{fezi98}.)
\subsection{Sets}
A subset $\Lambda \subseteq \Rdst$ is called \emph{relatively separated} if
the quantity,
\begin{equation}
\label{def_rel_sep}
\rel (\Lambda) := \max \set{\#(\Lambda \cap ([0,1]^d+x))}{x \in \Rdst}
\end{equation}
is finite. We will call the number $\rel (\Lambda)$ the \emph{relative separation} of
the set $\Lambda$. This is somehow an abuse of language since for a very separated set, this quantity is small.

A subset $\Lambda \subseteq \Rdst$ is called L-dense, for some $L>0$, if
\begin{equation}
\label{ldense}
\Rdst = \cup_{k\in\Lambda} B_L(k),
\end{equation}
where $B_L(k)$ denotes the open ball of center $k$ and radius $L$. $\Lambda$
is called \emph{relatively dense} if it is L-dense for some $L>0$.

\subsection{Weights}
A weight $w$ is a function $w:\Rdst \to (0,+\infty)$. For simplicity we will always assume that $w$ is continuous and symmetric.

A weight $w$ is said to be \emph{submultiplicative} if it satisfies,
\begin{equation}
\label{submult}
w(x+y) \leq w(x)w(y)
\mbox{, for all $x,y \in \Rdst$}.
\end{equation}
As an example, the \emph{polynomial weights}
\begin{equation}
\label{polyweight}
w_t(x) := (1+\abs{x})^t,
\end{equation}
satisfy the submultiplicativity condition if $t \geq 0$.

A second weight $v$ is called \emph{w-moderated} if it satisfies,
\begin{equation}
\label{moderated_weight}
v(x+y) \leq C v(x)w(y),
\end{equation}
for some constant $C>0$ and every $x,y \in \Rdst$. If the constant in
Equation \eqref{moderated_weight} is 1, we say that $v$ is \emph{strictly
moderated} by $w$.
The polynomial weight $w_t$ is strictly $w_s$-moderated if $s \geq 0$ and
$\abs{t} \leq s$.

For a sequence $\sett{c_k}_{k \in \Lambda} \subseteq \bC$, we consider the weighed norm,
\[
\norm{c}_{\ell^p_w} := \norm{d}_{\ell^p},
\mbox{where }
d_k := \abs{c_k} w(k),
\]
and we denote by $\ell^p_w$ the space of all such sequences having finite norm.
Weighed $L^p$ are defined similarly.

For technical reasons, we consider within $\ell^p_w$ the subspace $\el^p_w$ 
defined as the closure of the set of finitely-supported sequences. For $1
\leq p < \infty$ this is just the whole space $\ell^p_w$ and for $p=\infty$
it is $c^0_w$. We make this definition so as not to have to consider the case
$p=\infty$ separately.

We now state for future reference some facts about polynomial weights. The first lemma says that polynomial weights are \emph{subconvolutive} (see \cite{fe79}.)
\begin{lemma}
\label{conv_bound}
If $t>d$, then,
\[
w_{-t} * w_{-t} \leq K w_{-t},
\]
for some constant $K \lesssim \max \sett{1,1/(t-d)}$.
\end{lemma}
There is a corresponding statement for relatively separated index sets. The important point is that the bounds depend only on the relative separation of the sets involved (and this quantity is translation invariant.)
\begin{lemma}
\label{conv_nodes}
Let $\Gamma \subseteq \RRd$ be a relatively separated set of points and let $t>d$.
Then, the following estimates hold for a constant $K \lesssim \max
\sett{1,1/(t-d)}$.
\begin{itemize}
\item[(a)] $\sum_{\gamma \in \Gamma} w_{-t}(\gamma) \leq K \rel(\Gamma)$,
\item[(b)] $\sum_{\gamma: \abs{\gamma}>M} w_{-t}(\gamma) \leq K \rel(\Gamma) M^{-(t-d)}$,
\item[(c)] $\sum_{\gamma \in \Gamma} w_{-t}(\gamma) w_{-t}(x-\gamma)
\leq K \rel(\Gamma) w_{-t}(x)$, for all $x \in \Rdst$.
\end{itemize}
\end{lemma}
For proofs of very similar statements see \cite[Lemma 11.1.1]{gr01},
\cite[Lemma 1]{ro09} and \cite{fe79}.

\subsection{Amalgam spaces}
We denote by $\DRd$ the set of all $C^\infty$, compactly supported,
complex-valued functions on $\Rdst$, by $C^0(\Rdst)$ the set of all
continuous functions vanishing at infinity and by $\SchRd$ the Schwartz
class.

Let $\BspN$ be a \emph{uniformly localizable, isometrically translation invariant} Banach space. That is, $\Bsp$ satisfies the following axioms.
\begin{itemize}
 \item $\SchRd \hookrightarrow \Bsp \hookrightarrow \SchpRd$ are continuous
embeddings whose composition is the canonical embedding $\SchRd
\hookrightarrow \SchpRd$.
 \item If $h \in \DRd$ and $f \in \Bsp$, then $hf \in \Bsp$ and
there is a constant $C=C(h)>0$ such that $\norm{h(\cdot-x) f}_{\Bsp} \leq C
\norm{f}_{\Bsp}$, for all $x \in \Rd$ and $f \in \Bsp$.
 \item If $f \in \Bsp$ and $x \in \Rd$, then $f(\cdot-x) \in \Bsp$ and
$\norm{f(\cdot-x)}_{\Bsp}=\norm{f}_{\Bsp}$.
 \item Complex conjugation defines an isometry on $\Bsp$. That is,
if $f \in \Bsp$, then $\overline{f} \in \Bsp$ and $\norm{f}_\Bsp = \norm{\overline{f}}_\Bsp$.
\end{itemize}

We consider the space of distributions that belong to $\Bsp$ locally,
\begin{align*}
\Bsploc := \set{f \in \SchpRd}{hf \in \Bsp \mbox{, for all }h \in \DRd}.
\end{align*}
Given $f \in \Bsploc$ and a non-zero window $\eta \in \DRd$, we consider the
\emph{control function}
\begin{align*}
F(x):= \norm{f \eta(\cdot-x)},
\quad (x \in \Rdst.)
\end{align*}
For a space of functions $\Esp \hookrightarrow L^1_{\textit{loc}}$, the
Wiener amalgam space
$W(\Bsp,\Esp)$ is defined by
\begin{align*}
W(\Bsp,\Esp) := \set{f \in \Bsploc}{F \in \Esp},
\end{align*}
and is given the norm $\norm{f}_{W(\Bsp,\Esp)} := \norm{F}_\Esp$. The space
$\Esp$ is normally assumed to be \emph{solid} and \emph{translation
invariant}. The amalgam $W(\Bsp,\Esp)$ is then independent of the choice of
the window $\eta$ in the sense that different windows yield equivalent norms
(cf. \cite{fe83} and \cite{fe92-3}.)
In this article we will always let $\Esp$ be a weighed $L^p$ space.
\subsection{Amalgam norm of families}
We will consider a relatively separated set of points $\Lambda \subseteq
\Rd$, which will be called \emph{nodes} and a symmetric, submultiplicative,
continuous weight $w: \Rd \to (0,+\infty)$. We will also consider a family of
measurable functions $f_k: \Rd \to \bC$ indexed by the set of nodes
$\Lambda$.

For a family $F \equiv \sett{f_k}_{k \in \Lambda} \subseteq \Bsploc$ 
we define its $W(\Bsp,L^1_w)$ norm by
\begin{align*}
\norm{F}_{W(\Bsp, L^1_w)} := \max \sett{\sup_k \norm{g_k}_1, \supess_x \sum_k \abs{g_k(x)}},
\\
\mbox{where }
g_k(x) := \norm{f_k \eta(\cdot-x)}_{\Bsp} w(x-k),
\quad (x \in \Rdst, k \in \Lambda.)
\end{align*}
Here, $\eta \in \DRd$ is any nonzero window function (see Prop. \ref{norm_generalizes} below.)

Observe that if $\norm{F}_{W(\Bsp, L^1_w)} < +\infty$, then each $f_k$ belongs
to $W(\Bsp, L^1_w)$. The estimate $\norm{F}_{W(\Bsp, L^1_w)} < +\infty$
grants, in addition, certain uniformity for the set $\sett{f_k}_k$,
similar to that shared by the translates of an individual atom. Some results
to come will give evidence of that. The following proposition shows
that, at least, the hypothesis $\norm{F}_{W(\Bsp, L^1_w)} < +\infty$ indeed
extends to more general families $F$,
the condition $\norm{f}_{W(\Bsp, L^1_w)} < \infty$ normally imposed on families produced by translation of a single generator $f$. Before showing that, we must prove the independence of the window function in the definition above.

\begin{prop}
\label{norm_generalizes}
Let a family $F \equiv \sett{f_k}_{k \in \Lambda} \subseteq \Bsploc$
be given.
\begin{itemize}
 \item[(a)] Let $\norm{F}_{W_i(\Bsp, L^1_w)}$ be the norm defined using a nonzero window function $\eta_i \in \DRd$, (i=1,2). Then $\norm{F}_{W_1(\Bsp, L^1_w)} \approx \norm{F}_{W_2(\Bsp, L^1_w)}$.
 \item[(b)] For any bounded set $Q \subset \Rdst$ with non-empty interior,
the norm $\norm{F}_{W(\Bsp, L^1_w)}$ is also equivalent to the norm
 $\norm{F}_{\widetilde{W}(\Bsp, L^1_w)}$ defined by
 \begin{align*}
 \norm{F}_{\widetilde{W}(\Bsp, L^1_w)} &:= \max \sett{\sup_k \norm{g_k}_1, \supess_x \sum_k \abs{g_k(x)}},
 \\
 \mbox{where }
 g_k(x) &:= \norm{f_k}_{\Bsp(Q+x)} w(x-k),
 \quad x \in \Rd, k \in \Lambda,
 \\
 \mbox{and } \norm{f}_{\Bsp(Q)} &:= \inf \sett{\norm{g}_\Bsp: g \equiv f \mbox{ on Q}}.
 \end{align*}

 \item[(c)] If the family $F$ is given by $f_k = f(\cdot - k), k \in \Lambda$
and $\Lambda$ is relatively separated, 
then $\norm{F}_{W(\Bsp, L^1_w)} \approx \norm{f}_{W(\Bsp, L^1_w)}$.
\end{itemize}
\end{prop}
\begin{rem}
The implicit constant on (c) depends on the relative separation of $\Lambda$.
\end{rem}

\begin{proof}
For (a), since $\eta_2$ is compactly supported and not identically 0, it is possible to choose $\alpha>0$ such that $\sum_{j \in \Zdst} \abs{\eta_2}^2(\cdot- \alpha j) \approx 1$. This series is locally
finite, so the function $m := \eta_1 \left(\sum_{j \in \Zdst} \abs{\eta_2}^2(\cdot- \alpha j)\right)^{-1}$ is smooth.
Choose $\theta \in \DRd$ such that $\theta \equiv 1$ on the support of $\eta_1$.
Now, 
\begin{align*}
\eta_1=\theta \eta_1
= \sum_{j \in \Zdst} \theta m \abs{\eta_2}^2(\cdot- \alpha j). 
\end{align*}
Since both $\theta$ and $\eta_2$ are compactly supported, only finitely many terms are not zero and me may write
\[
\eta_1= \sum_{j=1}^n m_j \eta_2(\cdot-x_j),
\]
where $x_j \in \alpha \Zdst$ and $m_j := \theta m \overline{\eta_2(\cdot-x_j)} \in \DRd$.

Now, for $x \in \Rd$, and $k \in \Lambda$,
\begin{align*}
\norm{f_k \eta_1(\cdot-x)}_{\Bsp} w(x-k)
&\lesssim \sum_{j=1}^n \norm{f_k \eta_2(\cdot-x-x_j)}_{\Bsp} w(x-k)
\\
&\lesssim \sum_{j=1}^n \norm{f_k \eta_2(\cdot-(x+x_j))}_{\Bsp} w((x+x_j)-k) w(x_j)
\end{align*}
Consequently,
\[\norm{F}_{W_1(\Bsp,L^1_w)} \lesssim \norm{F}_{W_2(\Bsp,L^1_w)}.\]
The other inequality follows by symmetry.
\smallskip

To prove (b), consider first a window $\eta \in \DRd$ such that $\eta \equiv 1$ on $Q$.
Then for any $k \in \Lambda$ and $x \in \Rdst$, $\norm{f_k}_{\Bsp(Q+x)} \leq \norm{f_k \eta(\cdot-x)}_{\Bsp}$ and it follows that
$\norm{F}_{\widetilde{W}(\Bsp, L^1_w)} \lesssim
\norm{F}_{W(\Bsp, L^1_w)}$.

For the other inequality, since $Q$ has non-empty interior, there exists a non-zero window
function $\eta \in \DRd$ supported on $Q$. For any $k \in \Lambda$, $x \in \Rdst$
and any $h \in \Bsp$ such that $h \equiv f_k$ on $Q+x$, we have
\[
\norm{f_k \eta(\cdot-x)}_\Bsp = \norm{h \eta(\cdot-x)}_\Bsp \lesssim
\norm{h}_\Bsp.
\]
Therefore, $\norm{f_k \eta(\cdot-x)}_\Bsp \lesssim \norm{f_k}_{\Bsp(Q+x)}$, and the
desired inequality follows.

Let us now prove (c). For $x \in \Rd$ and $k \in \Lambda$, since
$\Bsp$ is isometrically translation invariant,
\[
g_k(x) = \norm{f(\cdot-k) \eta(\cdot-x)}_{\Bsp} w(x-k)
= \norm{f \eta(\cdot-(x-k))}_{\Bsp} w(x-k).
\]
Integrating over $x$ we get that for any $k \in \Lambda$,
\begin{equation}
\label{norm1_translation}
\norm{f}_{W(\Bsp, L^1_w)} = \norm{g_k}_1.
\end{equation}
This shows that $\norm{f}_{W(\Bsp, L^1_w)} \leq \norm{F}_{W(\Bsp, L^1_w)}$.

Since by \eqref{norm1_translation} we know that $\sup_k \norm{g_k}_1 \leq \norm{f}_{W(\Bsp, L^1_w)}$,
it suffices to show that $\supess_x \sum_k g_k(x) \lesssim \norm{f}_{W(\Bsp, L^1_w)}$.

To this end, let us call $Q$ the unitary cube centered at 0 and
let $\theta \in \DRd$ be such that $\theta \equiv 1$ on $\supp(\eta)+Q$.
For $x \in \Rd$, and $k \in \Lambda$,
\begin{align*}
g_k(x) &= \norm{f \eta(\cdot-(x-k))}_{\Bsp} w(x-k) 
= \int_Q \norm{f \eta(\cdot-(x-k))}_{\Bsp} w(x-k) dy
\\
&= \int_Q \norm{f \eta(\cdot-(x-k) \theta(\cdot-(x+y-k)}_{\Bsp} w(x-k) dy
\\
&\lesssim \int_Q \norm{f \theta(\cdot-(x+y-k))}_{\Bsp} w(x-k) dy
\\
&= \int_{Q+x-k} \norm{f \theta(\cdot-y)}_{\Bsp} w(x-k) dy
\end{align*}
Since $w$ is bounded on $Q$,
for $y \in Q+x-k$, $w(x-k) \leq w(y) \sup_Q w$.
Therefore, for any $x \in \Rd$,
\begin{align*}
\sum_k g_k(x) &\lesssim \sum_k \int_{Q+x-k} \norm{f \theta(\cdot-y)}_{\Bsp} w(y) dy
\\& = \int_\Rd \norm{f \theta(\cdot-y)}_{\Bsp} w(y) \sum_k \chi_{Q+x-k}(y) dy
\end{align*}
Finally, observe that $\sum_k \chi_{Q+x-k}(y)$ is bounded by the relative separation
of the set of nodes $\Lambda$. This completes the proof.
\end{proof}
\begin{example}
\label{poly_amalgam}
As an easy example of amalgam norm of families, consider a relatively separated set of nodes $\Lambda \subseteq \Rdst$, and a family of measurable functions $f_k: \Rdst \to \bC$, $k \in \Lambda$ satisfying the concentration condition,
\begin{equation}
\label{poly_example}
\abs{f_k(x)} \leq C w_{-(s+\alpha)}(x-k),
\quad x\in\Rdst, k \in \Lambda,
\end{equation}
for some $s>d$ and $\alpha \geq 0$.

Let $Q:=[0,1]^d$ be the unit cube. From equation \eqref{poly_example} we get that
for any $x \in \Rdst$,
\begin{align*}
\norm{f_k}_{L^\infty(Q+x)} \leq C \norm{w_{-(s+\alpha)}}_{L^\infty(Q+(x-k))}
\lesssim C w_{-(s+\alpha)}(x-k),
\end{align*}
where the implicit constant depends on $s+\alpha$. Therefore,
\begin{align*}
\norm{f_k}_{L^\infty(Q+x)} w_\alpha(x-k) \lesssim C w_{-s}(x-k).
\end{align*}
Hence by Proposition \ref{norm_generalizes} and Lemma \ref{conv_nodes},
$\norm{F}_{W(L^\infty,L^1_{w_\alpha})} \lesssim C \rel(\Lambda)$.
\end{example}
However, the concentration condition in Equation \eqref{poly_example} is much
more precise than the last statement. We will need this stronger condition
in Section \ref{sec_surgery}.
\subsection{Multiplier theorems}
\label{mult}
We now introduce a number of multiplier results that will replace in the
applications the convolution relations for amalgam spaces. These are easily
established for some endpoint spaces and then generalized by interpolation.
Throughout this section we will assume the following.
\begin{itemize}
 \item A relatively separated set of nodes $\Lambda \subseteq \Rd$ is given.
 \item $\Bsp$ is a uniformly localizable, isometrically translation invariant, Banach space.
 \item $w: \Rd \to (0, +\infty)$ is a symmetric, submultiplicative, continuous weight.
 \item $v: \Rd \to (0, +\infty)$ is a symmetric weight moderated by $w$.
\end{itemize}
We first show that the synthesis of well-localized atoms is bounded with respect to amalgam space norms.
\begin{prop}
\label{mult_synth}
Let a family $F \equiv \sett{f_k}_{k \in \Lambda} \subseteq \Bsploc$
such that $\norm{F}_{W(\Bsp, L^1_w)} < +\infty$
and $c \in \el^p_v$ be given ($1 \leq p \leq \infty$).
Then, the series
\[
c \cdot F := \sum_k c_k f_k,
\]
converges in $W(\Bsp, L^p_v)$ and satisfies the following estimate,
\[
\norm{c \cdot F}_{W(\Bsp, L^p_v)}
\lesssim \norm{c}_{\ell^p_v} \norm{F}_{W(\Bsp, L^1_w)}.
\]
\end{prop}
\begin{rem}
The implicit constant is the constant in \eqref{moderated_weight}.
\end{rem}
\begin{rem}
If $c \in \ell^\infty_v$, then the same conclusion holds but the series is only weak* convergent.
\end{rem}
\begin{proof}
We will assume that the sequence $c$ is finitely supported. The general case follows from this one by approximation, using the completeness of $W(\Bsp, L^1_v)$ (see \cite{fe83}.)

Let us set $f := c \cdot F = \sum_k c_k f_k$. For a window function $\eta \in \DRd$ and $x \in \Rd$, we have $f \eta(\cdot-x) = \sum_k c_k f_k \eta(\cdot-x)$. Therefore
\[
\norm{f \eta(\cdot-x)}_{\Bsp} v(x)
\leq C \sum_k \abs{c_k} v(k) \norm{f_k \eta(\cdot-x) }_{\Bsp} w(x-k),
\]
where the constant $C$ is the constant in \eqref{moderated_weight}.

Now Schur's lemma (see below) yields the desired inequality.
\end{proof}
In the proof we used part (a) of the following interpolation lemma which we quote for completeness. For a proof see \cite[Theorem 1.3.4]{gr04-3}.
\begin{lemma}
\label{schur}
Let $F \equiv \sett{f_k}_k$ be a family of measurable functions on $\Rdst$
and let $1 \leq p \leq \infty$.
\begin{itemize}
\item[(a)] Let $\sett{c_k}_k \subseteq \bC$ be a sequence. Then,
\[
\norm{c \cdot F}_{L^p} \leq \norm{c}_{\ell^p} \left( \sup_k \norm{f_k}_{L^1} \right)^{1/p}
\left( \supess_{x \in \RR^d} \sum_k \abs{f_k(x)} \right)^{1/p'}.
\]
\item[(b)] Let $g: \Rdst \to [0,\infty]$ be a measurable function. Then,
\[
\norm{g \cdot F}_{\ell^p} \leq \norm{g}_{L^p}
\left( \sup_k \norm{f_k}_{L^1} \right)^{1/p'}
\left( \supess_{x \in \RR^d} \sum_k \abs{f_k(x)} \right)^{1/p},
\]
where
\[
(g \cdot F)_k := \int_\Rdst g(x) f_k(x) dx.
\]
\end{itemize}
For both statements, if $p=1$, we interpret $1/p'=0$.
\end{lemma}

We now give estimates for transformations operating on families of
well-localized atoms. For a matrix of complex numbers $C \equiv
(c_{k,j})_{k,j \in \Lambda}$, we consider the following weighed Schur-type
norm,
\[
\norm{C}_\schurw :=
\max \sett{\sup_k \sum_j \abs{c_{k,j}} w(k-j), 
\sup_j \sum_k \abs{c_{k,j}} w(k-j)}.
\]
Furthermore, we denote by $\schurw$ the set of all such matrices having
finite norm.

Let us show that these matrices act boundedly on well-concentrated families
of atoms.
\begin{prop}
\label{mult_matrix}
Let a family $F \equiv \sett{f_k}_{k \in \Lambda} \subseteq \Bsploc$
such that $\norm{F}_{W(\Bsp, L^1_w)} < +\infty$
and a matrix $C \in \schurw$ be given.

Let $C \cdot F \equiv \sett{g_k}_k$ be the family defined by,
\[
g_k := \sum_j c_{k,j} f_j.
\]
Then, each of the series defining $g_k$ converges in $W(\Bsp, L^1_w)$
and we have the following estimate,
\[
\norm{C \cdot F}_{W(\Bsp, L^1_w)} \leq
\norm{C}_{\schurw} \norm{F}_{W(\Bsp, L^1_w)}.
\]
\end{prop}
\begin{proof}
Again, by an approximation argument we may assume that $C$ is finitely supported.

First observe that for fixed $k \in \Lambda$, the sequence $\sett{c_{k,j}}_j$
belongs to $\ell^1_m$, where $m$ is the weight given by $m(j) := w(k-j)$. Since $w(j) \leq m(j) w(k)$, it follows from Proposition \ref{mult_synth} that
the series defining $g_k$ converges in $W(\Bsp, L^1_w)$.

Fix a window function $\eta \in \DRd$ and $x \in \Rd$. For each $k \in \Lambda$,
$g_k \eta(\cdot-x) = \sum_j c_{k,j} f_j \eta(\cdot-x)$. Consequently, if we set
$h_k(x):=\norm{g_k \eta(\cdot-x)}_{\Bsp} w(x-k)$, we get,
\begin{align}
\label{mult_matrix_hk}
h_k(x)
&\leq
\sum_j \abs{c_{k,j}} w(k-j) \norm{f_j \eta(\cdot-x)}_{\Bsp} w(x-j).
\end{align}
Integrating this equation yields,
\[
\norm{h_k}_1 \leq \sum_j \abs{c_{k,j}} w(k-j) \norm{F}_{W(\Bsp,L^1_w)},
\]
so $\sup_k \norm{h_k}_1 \leq \norm{C}_\schurw \norm{F}_{W(\Bsp, L^1_w)}$.
From equation \eqref{mult_matrix_hk} we also get,
\begin{align*}
\sum_k h_k(x) &\leq
\sum_j \sum_k \abs{c_{k,j}} w(k-j) \norm{f_j \eta(\cdot-x)}_{\Bsp} w(x-j)
\\
&\leq \norm{C}_\schurw \sum_j \norm{f_j \eta(\cdot-x)}_{\Bsp} w(x-j).
\end{align*}
Therefore, $\supess_x \sum_k h_k(x) \lesssim \norm{C}_\schurw \norm{F}_{W(\Bsp,L^1_w)}$. This completes the proof.
\end{proof}
We now give a dual estimate.
\begin{prop}
\label{cross_correlation}
Let two families $F \equiv \sett{f_k}_{k \in \Lambda}, 
G \equiv \sett{g_k}_{k \in \Lambda} \subseteq \Bsploc$
such that $\norm{F}_{W(\Bsp,L^1_w)}, \norm{G}_{W(\Bsp,L^1_w)} < +\infty$
be given.
Suppose that $\Bsp$ is continuously embedded in $\Linfloc$. Then the cross-correlation matrix $C$, defined by
\[
C_{k,j} := \ip{f_k}{g_j},
\]
satisfies $\norm{C}_{\schurw} \lesssim \norm{F}_{W(\Bsp,L^1_w)} \norm{G}_{W(\Bsp,L^1_w)}$.
\end{prop}
\begin{rem}
The implicit constant depends on the embedding $\Bsp
\hookrightarrow \Linfloc$.
\end{rem}
\begin{proof}
Fix $\eta \in \DRd$ supported on an open ball $B$ around 0 and such that $\eta \equiv 1$
on a smaller concentric ball $B'$.
Let $f: \Rdst \to \bC$ be a locally integrable function.
Given $x \in \Rdst$, for almost every $y \in B'+x$,
\begin{align*}
\abs{f(y)} = \abs{f(y) \eta(y-x)} \leq \norm{f \eta(\cdot-x)}_{L^\infty(B)}
\lesssim \norm{f \eta(\cdot-x)}_\Bsp.
\end{align*}
Hence
\[
\abs{B_r(x)}^{-1} \int_{B_r(x)} \abs{f(y)} dy \lesssim \norm{f \eta(\cdot-x)}_\Bsp,
\]
for all sufficiently small $r>0$.
This shows that
\begin{align*}
\abs{f(x)} \lesssim \norm{f \eta(\cdot-x)}_\Bsp,
\end{align*}
at every $x \in \Rdst$ that is a Lebesgue point of $f$.
Consequently,
\begin{align*}
\abs{c_{kj}} w(k-j) \lesssim
\int_\Rdst \norm{f_k \eta(\cdot-x)}_\Bsp w(x-k)
\norm{g_j \eta(\cdot-x)}_\Bsp w(x-j) dx.
\end{align*}
Taking $\sup_k \sum_j$ and $\sup_j \sum_k$, it follows that
$\norm{C}_\schurw \lesssim \norm{F}_{W(\Bsp, L^1_w)} \norm{G}_{W(\Bsp, L^1_w)}$.
\end{proof}
Finally we show that well localized atoms induce bounded analysis operators.
\begin{prop}
\label{analysis}
Let a family $F \equiv \sett{f_k}_{k \in \Lambda} \subseteq \Bsploc$ 
such that $\norm{F}_{W(B,L^1_w)}<+\infty$ be given.
Suppose that $\Bsp$ is continuously embedded in $\Linfloc$.
For $f \in L^p_v$ ($1 \leq p \leq \infty$) define the analysis sequence,
\[
c_k := \ip{f}{f_k},
\quad k \in \Lambda.
\]
Then $c$ is well defined, belongs to $\ell^p_v$ and satisfies
\[
\norm{c}_{\ell^p_v} \lesssim \norm{f}_{L^p_v} \norm{F}_{W(B,L^1_w)}.
\]
\end{prop}
\begin{rem}
The implicit constant depends on the embedding $\Bsp \hookrightarrow
\Linfloc$ and the constant in \eqref{moderated_weight}.
\end{rem}
\begin{proof}
Fix $\eta \in \DRd$ supported on an open ball $B$ around 0 and such that $\eta \equiv 1$
on a smaller concentric ball $B'$. As in the proof of Proposition \ref{cross_correlation}, the sequence c satisfies
\[
\abs{c_k} \lesssim \int_{\Rdst} \abs{f(x)} \norm{f_k \eta(x-k)}_\Bsp dx,
\]
so
\[
\abs{c_k}v(k) \lesssim \int_{\Rdst} \abs{f(x)}v(x) \norm{f_k \eta(x-k)}_\Bsp w(x-k) dx.
\]
The conclusion now follows from part (b) of Lemma \ref{schur}.
\end{proof}
\subsection{Sets with multiplicity}
\label{set_mult}
For technical reasons, sometimes we will need to allow the set of nodes to
have repeated elements. A \emph{set with multiplicity} is simply a map
$\Gamma \ni \gamma \mapsto \gamma^* \in \Rdst$. Any subset of $\Rdst$ can be
considered as a set with multiplicity by letting the underlying map be the
inclusion. By abuse of notation, we will sometimes refer to a set with
multiplicity by the domain of the underlying map. For sets with multiplicity,
the relative separation is defined by,
\[
\rel (\Gamma) := \max_{x\in\Rdst} \# \sett{\gamma \in \Gamma / \gamma^* \in [0,1]^d+x}.
\]
Similarly, if a family of functions
$F \equiv \sett{f_\gamma}_{\gamma \in \Gamma} \subseteq \Bsploc$ is indexed by a set with multiplicity, the weighed-amalgam norm is defined by,
\begin{align*}
\norm{F}_{W(\Bsp, L^1_w)} := \max \sett{\sup_\gamma \norm{g_\gamma}_1, \supess_x \sum_\gamma \abs{g_\gamma(x)}},
\\
\mbox{where }
g_\gamma(x) := \norm{f_\gamma \eta(\cdot-x)}_{\Bsp} w(x-\gamma^*),
\quad x \in \Rd, \gamma \in \Gamma.
\end{align*}
Here, $\eta \in \DRd$ is any nonzero window function.

Every proof about ``ordinary'' relatively separated sets given in this
article also works for relatively separated sets with multiplicity. The
reader interested in this level of generality should read $\gamma^*$ whenever
an element $\gamma$ of a set with multiplicity is used as an element of the
Euclidean space (instead of as an index set.)

\section{Spline-type spaces}
\label{sec_splines}
We now formally introduce spline-type spaces. We consider a relatively separated set of points $\Lambda \subseteq \Rd$ which will be called \emph{nodes} and a family of functions $F \equiv \sett{f_k}_{k \in \Lambda} \subseteq \Linfloc$ that will be called \emph{atoms}.

Let $\splinezz$ be the set of finite linear combination of elements of $F$.
For a weight function $v$, and $1 \leq p \leq \infty$, we denote by
$\splinepv$ the $L^p_v$-closure of $\splinezz$. If the weight v is the trivial weight 1 we drop it in the notation.

Following the spirit of \cite{fegr89}, we do not want to consider each of the spaces
$\splinepv$ individually, but to treat all the range of spaces $\splinepv$ as a whole. We think of each $\splinepv$ as variant of a single spline-type space 
$\spline = \spline(F, \Lambda)$.

It will be assumed that the family $F$ is a Banach frame for each
$\splinepv$. The general theory of localized frames \cite{gr04-1, fogr05,
bacahela06} ensures that this is indeed the case provided that $F$ is a
Hilbert space frame for $\splinet$ and that $F$ satisfies a localization
property. In our context this property amounts to spatial localization and
the adequate technical tool was developed in \cite{grle06}
(see also \cite{su07-5}.)

We now formulate precisely the assumptions that we will make on the set of
atoms $F$ and show that under those assumptions, $F$ is a Banach frame for
the whole range of spaces $\splinepv$.

\begin{itemize}
 \item We will assume that we have chosen a uniformly localizable and isometrically translation invariant Banach space $\Bsp$, that is continuously embedded into $\Linfloc$. An example to keep in mind are fractional Sobolev spaces $\Sobqs$. These spaces are embedded in $\Linfloc$ if either $q=+\infty$ or if $s>d/q$ (see \cite{ad75}.)
 \item We will also assume that $F$ satisfies the uniform concentration and smoothness condition
 $\norm{F}_{W(\Bsp, L^1_w)} < +\infty$, for some weight $w: \Rd \to (0,\infty)$ that
 verifies,
 \begin{itemize}
  \item $w(x) := \exp{\rho(\norm{x})}$, for some norm $\norm{\cdot}$ on
$\Rdst$ and some continuous, concave function $\rho: [0, \infty] \to [0,
\infty]$ such that $\rho(0)=0$ and $\lim_{x \rightarrow +\infty}
\frac{\rho(x)}{x} = 0$,
  \item $w(x) \gtrsim (1+\norm{x})^\delta$, for some $\delta>0$.
 \end{itemize}
\item Finally, we will assume that $F$ forms a frame sequence in $\LtRd$.
\end{itemize}
If all the above assumptions are met we say that $\spline = \spline(F, \Lambda)$
is a spline type space.
\begin{rem}
Under the above assumptions, the weight $w$ satisfies: $w(0)=1$, $w(x)=w(-x)$ and is submultiplicative (see \cite{grle06}.)
\end{rem}
\begin{rem}
The polynomial weights $w_\alpha$ with $\alpha > 0$ and the subexponential
weights $w(x) := e^{\alpha\abs{x}^\beta}$ with $\alpha>0$ and $0<\beta<1$
satisfy the assumptions above (see \cite{grle06}.)
\end{rem}
The first items of the next proposition are a variation of \cite[Prop. 2.3]{fogr05}, adapted to our context.
\begin{theorem}
\label{spline_norm_equiv}
Let $\spline = \spline(F, \Lambda)$ be a spline type space, then the following holds.
\begin{itemize}
 \item[(a)] $G \equiv \sett{g_k}_k$, the canonical dual family of $F$ satisfies
$\norm{G}_{W(\Bsp , L^1_w)} < +\infty$.
 \item[(b)] For any $1 \leq p \leq \infty$ and any symmetric, $w$-moderated weight $v$, we have that $F$ is a Banach frame for $\splinepv$ with associated sequence space
$\el^p_v$.
 \item[(c)] For any $1 \leq p \leq \infty$ and any symmetric, $w$-moderated weight $v$, we have the inclusion $\splinepv \subseteq W(\Bsp , L^p_v)$. Moreover, on
$\splinepv$, the $L^p_v$ and $W(\Bsp , L^p_v)$ norms are equivalent.
\end{itemize}
\end{theorem}
\begin{rem}
In this situation, for $\ell^\infty$, the frame expansion arising from the
pair $(F,G)$ can be extended to the weak* closure of $\splinezz$ within
$L^\infty_v$, but the series converge only in the weak* topology. For details
see \cite{fegr89} and \cite{fogr05}.
\end{rem}
\begin{rem}
The norm equivalence of item (c) holds uniformly for $1 \leq p \leq +\infty$ and any class of $w$-moderated weights for which the constant in Equation \eqref{moderated_weight} is bounded.
\end{rem}
\begin{proof}
Consider the self-correlation matrix $C$, given by
$c_{kj} := \ip{f_k}{f_j}$. By proposition \ref{cross_correlation} we
have that $\norm{C}_\schurw < +\infty$. Since $F$ is a frame sequence in $\LtRd$, the matrix $C$ has a pseudo-inverse $D \in B(\ell^2)$. It follows from \cite[Theorem 3.4]{fogr05} that $\norm{D}_\schurw < +\infty$. Since $G = D \cdot F$, it follows from Proposition \ref{mult_matrix} that $\norm{G}_{W(\Bsp , L^1_w)} < +\infty$. This proves (a).

(b) is a slight variation of Proposition 2.3 in \cite{fogr05}. We only sketch the argument. Let $v$ be a symmetric, $w$-moderated weight and let $1 \leq p \leq \infty$. The reconstruction operator $R: \el^p_v \to \splinepv$, $c \mapsto c \cdot F$,
is well defined and bounded by Proposition \ref{mult_synth}. Moreover $\norm{R} \lesssim \norm{F}_{W(\Bsp , L^1_w)}$. Proposition \ref{analysis} implies that the coefficients mapping $C: L^p_v \to \ell^p_v$, given by $f \mapsto \sett{\ip{f}{g_k}}_k$ is well defined and satisfies $\norm{C} \lesssim \norm{G}_{W(\Bsp , L^1_w)}$. Moreover, if $f$ is a finite linear combination of functions of $F$, we have that $R C (f) = f$. It follows that $(F,G)$ determines a
Banach frame pair.

Now (c) follows easily from Proposition \ref{mult_synth}.
Since $\Bsp \hookrightarrow L^{\infty, \textit{loc}}
\hookrightarrow L^{p, \textit{loc}}$, we have the inclusion
$W(\Bsp, L^p_v) \hookrightarrow W(L^p, L^p_v)$.
Therefore, for $f \in \splinepv$, $f= RC(f)$ and,
\begin{align*}
\norm{f}_{L^p_v} &\lesssim \norm{f}_{W(\Bsp ,L^p_v)}
\lesssim \norm{F}_{W(\Bsp , L^1_w)} \norm{C(f)}_{\ell^p_v}
\\
&\lesssim \norm{F}_{W(\Bsp , L^1_w)} \norm{G}_{W(\Bsp , L^1_w)} \norm{f}_{L^p_v}.
\end{align*}
\end{proof}
We now observe that, in order to bound an operator on a spline-type space, we
just need to control its behavior on the atoms.
\begin{prop}
\label{bounds_operator_spline}
Let $\spline = \spline(F, \Lambda)$ be a spline-type space. Let
$v$ be a symmetric, $w$-moderated weight, $1 \leq p \leq \infty$ and let
$T: \splinepv \to L^p_v$ be a linear operator. Then,
\[
\norm{T}_{\splinepv \to L^p_v} \lesssim \norm{T(F)}_{W(\Bsp , L^1_w)}.
\]
\end{prop}
\begin{proof}
If $f=c \cdot F$ for a finitely supported sequence $c \in \ell^p(\Lambda)$,
then $T(f)=c \cdot T(F)$.
Theorem \ref{spline_norm_equiv} implies that,
\begin{align*}
\norm{T(f)}_{L^p_v} & \lesssim \norm{T(f)}_{W(L^\infty , L^p_v)}
\\
&\lesssim \norm{T(f)}_{W(\Bsp , L^p_v)}
\\
&\lesssim \norm{c}_{\ell^p_v} \norm{T(F)}_{W(\Bsp , L^1_w)}
\\
&\lesssim \norm{f}_{L^p_v} \norm{T(F)}_{W(\Bsp, L^1_w)}.
\end{align*}
The conclusion extends to general $f$ by an approximation argument.
\end{proof}

Finally we observe that, as a consequence of Theorem \ref{spline_norm_equiv}, there is a universal projector $P:L^p_v \to \splinepv$, for all $1 \leq p \leq \infty$ and $w$-moderated weights $v$. More precisely, we have the following statement.
\begin{theorem}
\label{universal_projector}
Let $\spline = \spline(F, \Lambda)$ be a spline-type space
and let $P: L^2 \to \splinet$ be the orthogonal projector onto $\splinet$.
Then, for
all $1 \leq p < \infty$ and $w$-moderated weights $v$, the restriction of $P$
to $\SchRd$ extends by density to a bounded projector $P: L^p_v \to \splinepv$. For
$p=\infty$ the same statement is true replacing $L^\infty_v$ for $C^0_v$.

Moreover, the norm of $P$ is uniformly bounded for $1 \leq p \leq +\infty$ and any class of $w$-moderated weights for which the constant in Equation \eqref{moderated_weight} is bounded.
\end{theorem}
\begin{proof}
We only need to check that the restriction of $P$ to $\SchRd$ is bounded in
the norm of $L^p_v$. The projector $P$ is given by
\[
P(f) = \sum_{k \in \Lambda} \ip{f}{g_k} f_k,
\]
where $G$ is the family of dual atoms given by Theorem \ref{spline_norm_equiv} (a). Using Propositions \ref{mult_synth} and \ref{analysis} we get,
\begin{align*}
\norm{P(f)}_{L^p_v} &\lesssim \norm{P(f)}_{W(\Bsp, L^p_v)}
\lesssim \norm{F}_{W(\Bsp, L^1_w)} \norm{G}_{W(\Bsp, L^1_w)} \norm{f}_{L^p_v}.
\end{align*}
\end{proof}
\section{Frame surgery}
\label{sec_surgery}
In this section we consider the following locality problem. We are given a
spline type space $\spline$ and several exterior frames
$\sett{\varphi^i_k}_{k \in \Lambda_i}$, $i \in I$,
for it. For each of these frames, we arbitrarily select a region of the euclidean space $\parti$ where we want to use it. The family $\sett{\parti}_i$ must form a covering of $\Rdst$. We argue that, if for each $i \in I$ we pick from the frame 
$\sett{\varphi^i_k}_{k \in \Lambda_i}$ those elements that are concentrated
near $\parti$, then the resulting family
$\sett{\varphi^i_k}_{k \in \Delta_i, i \in I}$
forms an exterior frame for $\spline$. Moreover, given (possibly
non-canonical) dual frames for each of the original exterior frames, we
provide an approximate reconstruction operator for the new exterior frame.

Since we are not dealing with frames for the whole space $L^2(\Rdst)$, we
cannot take a function $f$, break it into pieces $f_i$ supported on $\parti$,
expand each $f_i$ using the exterior frame $\sett{\varphi^i_k}_k$, and then
add all those expansions. This approach does not work in our context because
for $f \in \spline$, the localized pieces $f_i$ do not belong to $\spline$
and consequently cannot be expanded using the frame $\sett{\varphi^i_k}_k$.

Instead, we argue that for each member of the covering $\parti$, the norm of
a function $f \in \spline$ restricted to $\parti$ should depend mainly on the
atoms concentrated around $\parti$. Then we glue these local estimates by
means of an almost-orthogonality principle, which is implicit in the
computations below.

To be able to quantify the approximation scheme we will work with
frames polynomially localized in space.

\subsection{The approximation scheme}
\label{gluying}
We now give the precise assumptions for this section.

\begin{itemize}
\item We assume that $\spline = \spline(F, \Lambda)$ is a spline-type space
where the atoms $F \equiv \sett{f_k}_{k \in  \Lambda}$ and a given system of
dual atoms 
$G \equiv \sett{g_k}_{k \in  \Lambda}$ satisfy,
\begin{equation}
\label{frame_loc}
\abs{f_k(x)}, \abs{g_k(x)} \leq C \left(1+\abs{x-k}\right)^{-(s+\alpha)}
\qquad (x \in \Rdst, k\in\Lambda),
\end{equation}
for some constants $C>0$, $s>d$ and $\alpha \geq 0$. It is well-known
\cite{gr04-1} that if this condition holds for the atoms $F$, then it is
automatically satisfied by some system of dual atoms $G$.

\item We are given a family of frame pairs for $\splinet$.
\footnote{Remember that the analyzing atoms $\sett{\varphi^i_k}_k$ need not
to belong to
$\splinet$.}
,
\[
\left( \sett{\psi^i_k}_{k \in \Lambda_i}, \sett{\varphi^i_k}_{k \in \Lambda_i} \right)
\qquad (i \in I),
\]
that satisfy the following uniform concentration condition around their nodes $\Lambda_i$,
\begin{equation}
\label{frames_loc_2}
\abs{\varphi^i_k(x)}, \abs{\psi^i_k(x)} \leq C \left(1+\abs{x-k}\right)^{-(s+\alpha)}
\qquad (x \in \Rdst, k\in\Lambda_i, i \in I),
\end{equation}
for some constant $C>0$, that, for simplicity, is assumed to be equal to the
constant in \eqref{frame_loc}.

Observe that we are requiring all the frames \emph{and the dual frames} to
be uniformly localized. Given a concrete family of uniformly localized
(exterior) frames, it can be difficult to decide if they possess a
corresponding family of dual frames sharing a common spatial localization.
This problem is addressed in Section \ref{sec_quality}.

\item We have a (measurable) covering of $\Rdst$, $\partt \equiv
\sett{\parti}_{i \in I}$ that is \emph{uniformly locally finite}. This means
that for some (or any) cube Q,
\begin{equation}
\label{loc_finite_partition}
\partnq :=
\max_{x\in\Rdst} \# \sett{i \in I / (Q+x) \cap \parti \not= \emptyset} < \infty.
\end{equation}
Observe that this assumption in particular implies that number of overlaps of $\partt$ is finite. That is,
\begin{equation}
\label{bounded_overlaps}
\partn :=
\max_{x\in\Rdst} \# \sett{i \in I / x \in \parti} < \infty.
\end{equation}

\item We suppose that the set of nodes $\Lambda_i$ are \emph{uniformly relatively separated}. That is,
\begin{equation}
\label{unif_rel_sep}
\sup_{i \in I} \rel (\Lambda_i) < \infty.
\end{equation}
Observe that this assumption, together with the uniform localization of the dual frames, implies that the original frames have a uniform common lower bound.
\end{itemize}
We now prove the central approximation result. In Section \ref{new_frames} we apply this result to the construction of new frames.
\begin{theorem}
\label{glue}
Let $\sett{\partfi}_{i \in I}$ be a (measurable) partition of the unity
subordinated to $\partt$. That is, each $\partfi$ is nonnegative, $\sum_i
\partfi \equiv 1$ and
$\supp(\partfi) \subseteq \parti$.

For each $r>0$ consider the sets
\begin{equation*}
\Lambda_i^r := \sett{k \in \Lambda_i / d(k,\parti) \leq r}
\end{equation*}
and choose any set $\Delta_i^r$ such that $\Lambda_i^r \subseteq \Delta_i^r \subseteq \Lambda_i$.
Consider also, the following \emph{approximate reconstruction operator},
\[
A^r(f) := \sum_{i \in I} \sum_{k \in \Delta_i^r} \ip{f}{\varphi^i_k} \psi^i_k \partfi.
\]

Then, for any $1 \leq p \leq \infty$, any $w_\alpha$-moderated weight $v$
and any $f \in \splinepv$,
\[
\norm{A^r(f)-f}_{L^p_v} \leq K \partn \norm{f}_{L^p_v} r^{-(s-d)},
\]
where $K>0$ is a constant that only depends on $d, C, s, \alpha$, the set of nodes $\Lambda$, the common relative separation of all the sets of nodes $\Lambda_i$ and the constant in Equation \eqref{moderated_weight} using the weight $w_\alpha$ as moderator.
\end{theorem}
\begin{rem}
The fact that the covering is uniformly locally finite is not used in the proof of the theorem; only the weaker condition in Equation \eqref{bounded_overlaps}
is needed. However, the stronger assumption of Equation \eqref{loc_finite_partition} is required for the applications of Section \ref{sec_applications}.
\end{rem}
\begin{proof}
Observe first that we can always add more nodes to the set $\Lambda$ and
extend the set of atoms $F$ and dual atoms $G$ by associating 0 to the new
nodes. All the assumptions on the atoms are preserved by this extension, but
the relative separation of the set of nodes changes. By adding to $\Lambda$
any fixed relatively separated and relatively dense set $\Gamma$, we can
assume that $\Lambda$ is L-dense (c.f. Equation \eqref{ldense}.) The relative
separation of the resulting set can be bounded by
$\rel(\Lambda)+\rel(\Gamma)$.

For all $i \in I$, every $f \in \splinezz$ admits the expansion,
\[
f = \sum_{j \in \Lambda_i} \ip{f}{\varphi^i_j} \psi^i_j.
\]
Averaging all these expansions yields,
\[
f = \sum_{i \in I} \sum_{j \in \Lambda_i} \ip{f}{\varphi^i_j} \psi^i_j \partfi.
\]
Since $f$ also admits the expansion $f=\sum_k \ip{f}{g_k} f_k$, it follows that
\[
f = \sum_{k \in \Lambda} \ip{f}{g_k}
\sum_{i \in I} \sum_{j \in \Lambda_i} \ip{f_k}{\varphi^i_j} \psi^i_j \partfi.
\]

Similarly,
\begin{align*}
A^r(f) &= \sum_{i \in I} \sum_{j \in \Delta^r_i} \ip{f}{\varphi^i_j} \psi^i_j \partfi
\\
&= \sum_{k \in \Lambda} \ip{f}{g_k} \sum_{i \in I} \sum_{j \in \Delta^r_i} \ip{f_k}{\varphi^i_j} \psi^i_j \partfi
\end{align*}
Therefore, for $f \in \splinezz$,
\begin{align*}
f-A^r(f) &= \sum_{k \in \Lambda} \ip{f}{g_k} \sum_{i \in I} \sum_{j \notin \Delta^r_i} \ip{f_k}{\varphi^i_j} \psi^i_j \partfi
\end{align*}
Consequently, if we set $c_k := \ip{f}{g_k}$, by Lemma \ref{conv_bound},
\begin{align*}
\abs{f-A^r(f)}
&\lesssim \sum_{k \in \Lambda} \abs{c_k} \sum_{i \in I}
\sum_{j \notin \Delta_i} \ip{w_{-(s+\alpha)}(\cdot-k)}{w_{-(s+\alpha)}(\cdot-j)} w_{-(s+\alpha)}(\cdot-j) \chi_{\parti}
\\
&\lesssim \sum_{k \in \Lambda} \abs{c_k} \sum_{i \in I}
\sum_{j \notin \Lambda^r_i} w_{-(s+\alpha)}(k-j)
 w_{-(s+\alpha)}(\cdot-j) \chi_{\parti}.
\end{align*}
If we define,
\[
E^r_k := \sum_{i \in I}
\sum_{j \notin \Lambda^r_i} w_{-(s+\alpha)}(k-j)
 w_{-(s+\alpha)}(\cdot-j) \chi_{\parti},
\]
Lemma \ref{schur} implies that, 
\begin{equation*}
\norm{f-A^r(f)}_{L^p_v} \lesssim \norm{c}_{\ell^p_v} \left( \sup_k \norm{E^r_k w_\alpha(\cdot-k)}_1 \right)^{1/p}
\left( \supess_{x \in \RR^d} \sum_k \abs{E^r_k(x) w_\alpha(x-k)} \right)^{1/p'}.
\end{equation*}
Since $(F,G)$ is a frame pair for $\splinepv$, $\norm{c}_{\ell^p_v} \leq K
\norm{f}_{L^p_v}$, for some constant $K$ that only depends on d, C, s,
$\alpha$,
the constant in Equation \eqref{moderated_weight} (taking $w=w_\alpha$) and $\Lambda$. Consequently,
\begin{equation}
\label{error_bound}
\norm{f-A^r(f)}_{L^p_v} \leq K \norm{f}_{L^p_v} \left( \sup_k \norm{E^r_k w_\alpha(\cdot-k)}_1 \right)^{1/p}
\left( \supess_{x \in \RR^d} \sum_k \abs{E^r_k(x) w_\alpha(x-k)} \right)^{1/p'}.
\end{equation}

Now observe that, since $w_\alpha(x-k) \leq w_\alpha(x-j) w_\alpha(k-j)$,
\begin{align*}
\abs{E^r_k(x) w_\alpha(x-k)} &\leq
\sum_{i \in I} \sum_{j \notin \Lambda^r_i} w_{-s}(k-j)
 w_{-s}(x-j) \chi_{\parti}(x).
\end{align*}
For every $i \in I$, since $\Lambda$ is now assumed to be $L$-dense, there
exists a map $\mu_i: \Lambda_i \to \Lambda$ such that $\abs{k-\mu_i(k)} \leq
L$, for all $k \in \Lambda_i$. This map will be used to reduce the proof to
the case where all the index sets are equal. This same argument was used in
\cite{bacahela06}, where irregularly distributed phase-space points are
assigned a near point in a regular reference system by means of a `round-up'
map.

Since the sets $\Lambda_i$ are assumed to be uniformly relatively
separated, 
there exists a number $N \in \Nst$, that depends only on $L$ and the relative separation of all the sets of nodes, such that
\begin{align*}
\# \mu_i^{-1}(\sett{j}) \leq N
\mbox{, for every $j \in \Lambda$}.
\end{align*}

Suppose initially that $r > 2L$, define $R:=r-L$ and estimate,
\begin{align*}
\abs{E^r_k(x) w_\alpha(x-k)}
&\leq 
\sum_{i \in I} \sum_{j \in \Lambda}
\mathop{\sum_{l \in \mu_i^{-1}(j),}}_{l \notin \Lambda^r_i}
w_{-s}(k-l) w_{-s}(x-l) \chi_{\parti}(x).
\end{align*}
If $l \in \mu_i^{-1}(j)$, then $\abs{j-l} \leq L$, 
so $w_{-s}(k-l) \lesssim w_{-s}(k-j)$ and
$w_{-s}(x-l) \lesssim w_{-s}(x-j)$. (Here the implicit constants depend on $L$ and $s$.)
If in addition $l \notin \Lambda^r_i$, then $j \notin \Omega_i^R$,
where
\begin{equation*}
 \Omega^R_i := \set{k \in \Lambda}{d(k,E_i) \leq R}.
\end{equation*}
Consequently,
\begin{align}
\label{error_kernel_estimate}
\abs{E^r_k(x) w_\alpha(x-k)}
&\lesssim 
N \sum_{i \in I} \sum_{j \notin \Omega^R_i}
w_{-s}(k-j) w_{-s}(x-j) \chi_{\parti}(x).
\end{align}
Using this estimate, we bound the weighed Schur norm of the kernel $E^r$.

For every $k \in \Lambda$,
\begin{align*}
\norm{E^r_k w_{\alpha}(\cdot-k)}_1
&\lesssim
\sum_{i \in I} \sum_{j \notin \Omega^R_i}
w_{-s}(j-k) \int_{\parti} w_{-s}(x-j) dx
\\
&= 
\sum_{j \in \Lambda} \sum_{\substack{i \in I \\ d(j, \parti)>R}}
w_{-s}(j-k) \int_{\parti} w_{-s}(x-j) dx
\\
&\leq
\partn \sum_{j \in \Lambda} w_{-s}(j-k)
\int_{\bigcup \parti} w_{-s}(x-j) dx,
\end{align*}
where the union in the last integral ranges over all $i \in I$ such that $d(j, \parti)>R$.
Since the complement of the cube $Q_R(j)$ contains that union, we get,
\begin{align*}
\norm{E^r_k w_{\alpha}(\cdot-k)}_1 &\leq
\partn \sum_{j \in \Lambda} w_{-s}(j-k)
\int_{\RRd \setminus Q_R(j)} w_{-s}(x-j) dx
\\
&=
\partn \sum_{j \in \Lambda} w_{-s}(j-k)
\int_{\RRd \setminus Q_R(0)} w_{-s}(x) dx.
\end{align*}
The set $\Lambda-k$ has the same relative separation
that $\Lambda$, so Lemma \ref{conv_nodes} implies that
\begin{align*}
\norm{E^r_k w_{\alpha}(\cdot-k)}_1 &\lesssim
\partn \int_{\RRd \setminus Q_R(0)} w_{-s}(x) dx
\\
&\lesssim \partn R^{-(s-d)}.
\end{align*}
Since $r>2L$, it follows that $R>r/2$ and 
\begin{equation}
\label{first_error_bound}
\sup_k \norm{E^r_k w_\alpha(\cdot-k)}_1 \lesssim \partn r^{-(s-d)}.
\end{equation}
Using again the estimate in Equation \eqref{error_kernel_estimate},
we bound now $\supess_x \sum_k \abs{E^r_k(x) w_\alpha(x-k)}$.

Fix $x \in \RRd$ and let
\begin{align*}
I_x := \set{i \in I}{x \in \parti}.
\end{align*}
From Equation \eqref{bounded_overlaps} we know that $\# I_x \leq \partn$.
We now estimate,
\begin{align*}
\sum_{k \in \Lambda} \abs{E^r_k(x) w_\alpha(x-k)} &\lesssim
\sum_{i \in I_x} \sum_{j \notin \Omega^R_i} \sum_{k \in \Lambda}
w_{-s}(k-j) w_{-s}(x-j).
\end{align*}

Since $\Lambda$ and $\Lambda-\sett{j}$ have the same relative separation,
Lemma \ref{conv_nodes} implies that,
\begin{align*}
\sum_{k \in \Lambda} w_{-s}(k-j) \lesssim 1,
\end{align*}
so,
\begin{align*}
\sum_{k \in \Lambda} \abs{E^r_k(x) w_\alpha(x-k)} &\lesssim
\sum_{i \in I_x} \sum_{j \notin \Omega^R_i} w_{-s}(x-j).
\end{align*}
For $i \in I_x$ and $j \notin \Omega^R_i$, we have that 
$\abs{x-j} \geq d(j,E_i) > R$. It follows that,
\begin{align*}
\sum_{k \in \Lambda} \abs{E^r_k(x) w_\alpha(x-k)} 
&\leq
\sum_{i \in I_x} \sum_{j: \abs{x-j} > R} w_{-s}(x-j).
\end{align*}
Since the sets $\Gamma$ and $x-\Gamma$ have the same relative separation, Lemma \ref{conv_nodes} implies that,
\begin{align*}
\sum_{k \in \Lambda} \abs{E^r_k(x) w_\alpha(x-k)} & \lesssim \partn R^{-(s-d)}.
\end{align*}
Using again the fact that $r>2L$, it follows that,
\begin{equation}
\label{second_error_bound}
\supess_{x \in \RRd} \sum_{k \in \Lambda} \abs{E^r_k(x) w_\alpha(x-k)}
\lesssim
\partn r^{-(s-d)}.
\end{equation}
Combining the estimates in Equations \eqref{first_error_bound},
\eqref{second_error_bound} and \eqref{error_bound}, it follows that
\[
\norm{A^r(f)-f}_{L^p_v} \lesssim \norm{f}_{L^p_v} \partn r^{-(s-d)},
\]
for $r>2L$.

It remains to show that a similar estimate holds for $0<r \leq 2L$. In this case,
$r^{-(s-d)} \gtrsim 1$. So, it suffices to observe that $\norm{A^r}_{\splinepv \to L^p_v} \lesssim \partn$, uniformly on $r$. Reexamining the estimates given for the error kernel $E^r$, the desired conclusion follows.
\end{proof}
\begin{rem}
The technique in the proof of the theorem of using the frame expansion twice
is somehow analogous to the use of reproducing formulas in the classical
decomposition results for function spaces (see for example \cite{frjawe91}
and \cite{fegr89}.)

The formula defining the operator $A^r$ makes sense in $L^p_v$, but the
bound given in the theorem is only valid in the smaller subspace $\splinepv$,
where the ``reproducing formula'' (the frame expansion) is valid. By means of
it, the task of bounding the operator is reduced in the proof to the one of
controlling its behavior on atoms, much in the spirit of the classical
atomic decompositions (see \cite{frjawe91} and also \cite{grpi09}.)
\end{rem}
\subsection{Constructing new frames}
\label{new_frames}
We now interpret the approximation result of Section \ref{gluying} as a
method to produce new frames. Observe that, however, for some applications,
the estimate provided by Theorem \ref{glue} is all that is needed. If
concrete atoms and dual atoms are known, then the estimate in the theorem
provides an approximate reconstruction operator for the new system of atoms.

Consider again the ingredients of Section \ref{gluying} and
let $\sett{\partfi}_{i \in I}$ be a (measurable) partition of the unity
subordinated to $\partt$ (e.g. $\partfi = (\sum_j \chi_{\partj})^{-1}
\chi_{\parti}$.)

Let $v$ be a $w_\alpha$-moderated weight and let $P: L^p_v \to \splinepv$ be the universal projector onto $\splinepv$ (cf. Theorem \ref{universal_projector}.)

Fix a value of $r>0$ and consider the operator $B^r: \splinepv \to \splinepv$ given by $B^r := P \circ A^r$,
where
\[
A^r(f) := \sum_{i \in I} \sum_{k \in \Lambda_i^r} \ip{f}{\varphi^i_k}
\psi^i_k \partfi,
\]
and, as before, $\Lambda_i^r := \sett{k \in \Lambda_i / d(k,\parti) \leq r}$.

For each $i \in I$, let $\splinepvi$ be the $L^p_v$-closed linear space
generated by the atoms $\sett{\psi^i_k \partfi}_{k \in \Lambda_i^r}$. These spaces, of course, depend on $r$.

Consider the direct sum $\oplus_i \splinepvi$ as a subspace of
$\ell^p_{L^p_v}$, the space of $L^p_v$-valued $\ell^p$ families;
more precisely, $\oplus_i \splinepvi$ is the closure of the algebraic direct
sum within $\ell^p_{L^p_v}$. Let $\iota:
\oplus_i \splinepvi \to L^p_v$ be the operator given by
\begin{align*}
\iota( (f^i)_i ) := \sum_i f^i.
\end{align*}
Since $\partt$ is locally finite, $\iota$ is well-defined and bounded uniformly on
$p$ and $v$. Indeed,
for $1 \leq p < \infty$,
\begin{align*}
\norm{\sum_i f^i}_{L^p_v}^p &=
\int_\Rdst \abs{\sum_i f_i(x)}^p v(x)^p dx
\\
&\leq \int_\Rdst (\sum_{i\in I_x} \abs{f_i(x)})^p v(x)^p dx,
\end{align*}
where $I_x := \set{i \in I}{x \in \parti}$. Since $\# I_x \leq \partn$,
\begin{align*}
\norm{\iota((f^i)_i)}_{L^p_v}^p
&\leq \partn^p \int_\Rdst \sum_i \abs{f_i(x)}^p v(x)^p dx
\\
&= \partn^p \sum_i \norm{f_i}_{L^p_v}^p
\\
&= \partn^p \norm{(f^i)_i)}_{\ell^p_{L^p_v}}^p.
\end{align*}
So, $\norm{\sum_i f^i}_{L^p_v} \leq \partn \norm{(f^i)_i)}_{\ell^p_{L^p_v}}$. For $p=\infty$, a similar computation establishes the same estimate. Composing $\iota$ with the projector $P$, we get a \emph{synthesis operator} $Sy: \oplus_i \splinepvi \to \splinepv$.

For each $i \in I$, let $Q_i:\splinepv \to \splinepvi$ be given by 
\[
Q_i(f) := \sum_{k \in \Lambda_i^r} \ip{f}{\varphi^i_k} \psi^i_k \partfi. 
\]
The concentration conditions on Equation \eqref{frames_loc_2} imply that
all these operators are uniformly bounded. Moreover, they determine a map
$Q: \splinepv \to \oplus_i \splinepvi$, given by $Q(f) := \left( Q_i(f) \right)_i$. We will prove below that $Q$ is well defined and bounded. Assuming this for the moment, we have a commutative diagram,
\begin{equation}
\label{fusion_frame}
\xymatrix{ 
        \splinepv \ar[dr]_{B^r} \ar[r]^Q & \bigoplus_i \splinepvi \ar[d]^{Sy}\\ 
        & \splinepv }
\end{equation}
It follows from Theorem \ref{glue} that for a sufficiently large value of
$r>0$, $B^r$ is invertible and consequently Q is left-invertible and Sy is
right-invertible. This provides two ways of viewing $\splinepv$ as a retract
of $\oplus_i \splinepvi$. One is $Q$ (with retraction $(B^r)^{-1} Sy$) and
the other is $Q (B^r)^{-1}$ (with retraction Sy.) In the spirit of
\cite{os97} and \cite{caku04}, this can be called an \emph{exterior Banach
fusion frame} or an \emph{exterior stable splitting} (see also
\cite{fegr85} and \cite{fe87}.)

Now observe that each of the maps $Q_i$ can be factored through $\el^p_v$,
\[
\xymatrix{ 
        \splinepv \ar[dr]_{C_i} \ar[r]^{Q_i} & \splinepvi\\ 
        & \el^p_v(\Lambda^r_i) \ar[u]_{R_i} }
\]
where $C_i(f) := \left( \ip{f}{\varphi^i_k} \right)_k$ and
$R_i(c) := \sum_k c_k \psi^i_k \partfi$.

This induces a commutative diagram,
\[
\xymatrix{ 
        \splinepv \ar[dr]_C \ar[r]^{Q} & \bigoplus_i \splinepvi\\ 
        & \bigoplus_i \el^p_v(\Lambda^r_i) \ar[u]_R }
\]
where in $\bigoplus_i \el^p_v(\Lambda^r_i)$ we use the p-norm; that is
$\norm{(c^i)_i} := \norm{(\norm{c^i}_{\ell^p_v})_i}_{\ell^p}$. This is just a weighed $\ell^p$ norm; this way of presenting it is due to the structure of the index sets.
The boundedness of the operators $C$ and $R$ is proved in Theorem
\ref{gluying_frames} below. Assuming this fact for the moment, observe that
if $B^r$ is invertible, then Q is left-invertible and so is $C$. We formalize
this in the following Theorem.
\begin{theorem}
\label{gluying_frames}
Suppose that the assumptions of Section \ref{gluying} are satisfied. Let $v$
be a $w_\alpha$-moderated weight. Then for all sufficiently large values of
$r>0$,
\[
\sett{\varphi^i_k: i \in I, k \in \Lambda^r_i}
\]
is a Banach frame for $\splinepv$.

More precisely, if we define the index set $\Gamma := \bigcup_{i \in I} \Lambda^r_i \times \sett{i}$ and the weight $V(k,i) := v(k)$, then the analysis map
\begin{align*}
\splinepv &\to \el^p_V(\Gamma)
\\
f &\mapsto \left(\ip{f}{\varphi^i_k}\right)_{(k,i)}
\end{align*}
is bounded and left-invertible, for all sufficiently large values of $r>0$.

Moreover, the value of $r$ may be chosen uniformly for all $1 \leq p \leq \infty$
and every class of $w_\alpha$-moderated weights for which the respective constant (cf. Equation \eqref{moderated_weight}) is uniformly bounded.
\end{theorem}
\begin{rem}
Observe that although we are constructing a new frame $\sett{\varphi^i_k}_{i
\in I, k \in \Lambda^r_i}$ out of the pieces $\sett{\varphi^i_k}_{k \in
\Lambda^r_i}$, we do not claim that each of these pieces forms a frame
sequence. This construction should be compared to the methods in
\cite{alcamo04}, \cite{fo04} and \cite{caku04} where a global frame is
built \emph{from} local (possibly exterior) frames for certain subspaces.
\end{rem}
\begin{rem}
As a related result, we mention Lemma 4.7 in \cite{rosh97-1} where it is shown that
if $\set{2^{k/2} \psi(2^k \cdot -j)}{k,j \in \Zst}$ is a wavelet frame for
$L^2(\Rst)$ and the wavelet $\psi$ satisfies a mild smoothness condition,
then for all sufficiently large values of $r>0$, the system of fine scales
$\set{2^k \psi(2^k \cdot +j)}{k \in \Zst, j \geq 0}$
forms an exterior frame for the subspace
$H_r := \set{f \in L^2(\Rst)}{\hat{f} \equiv 0 \mbox{ on } [-r,r]}$.
\end{rem}
\begin{proof}
Using Theorem \ref{glue}, choose a value of $r>0$ such that the operator $B^r$ is invertible. By the discussion above, it only remains to bound the operators $C$ and $R$. Consider the index set $\Gamma$ as a set with multiplicity (cf. Section \ref{set_mult}), where we map each of the sets $\Lambda^r_i \times \sett{i}$ into $\Rdst$ by discarding the second coordinate.

The fact that the sets $\Lambda_i$ are uniformly relatively separated 
(cf. Equation \eqref{unif_rel_sep}) and that the covering $\partt$ is
uniformly locally finite (cf. Equation \eqref{loc_finite_partition}) implies
that $\Gamma$ is relatively separated. Indeed, let $Q$ be the unit cube and
$Q':=Q+[-r,r]^d$. For any $x \in \Rdst$,
\begin{align*}
\set{(k,i) \in \Gamma}{k \in Q+\sett{x}}
&= \bigcup_{i \in I} \set{k \in \Lambda_i \cap (Q+\sett{x})}{d(k,E_i) \leq r} \times \sett{i}
\\
&\subseteq \bigcup_{i \in I_x} \left( \Lambda_i \cap (Q+\sett{x}) \right) \times \sett{i},
\end{align*}
where $I_x := \set{i \in I}{E_i \cap \left(Q'+\sett{x}\right) \not= \emptyset}$.
Hence $\rel(\Gamma) \leq \partnqp \sup_i \rel(\Lambda_i) < \infty$ (cf.
Equation \eqref{loc_finite_partition}.)

The family of atoms $\sett{\varphi^i_k}_{(k,i) \in \Gamma}$ satisfies a polynomial concentration condition. By Example \ref{poly_amalgam}, the family has finite $W(L^\infty,L^1_{w_\alpha})$ norm. The boundedness of the operator $C$ now follows from Propositions \ref{analysis}.

For the boundedness of $R$, observe that the families $\sett{\psi^i_k \partfi}_{k \in \Lambda^r_i}$ satisfy a uniform polynomial concentration condition and their nodes are uniformly relatively separated. Hence, by Example \ref{poly_amalgam},
\[
M:=\sup_i \norm{\sett{\psi^i_k \partfi}_k}_{W(L^\infty,L^1_{w_\alpha})} < \infty.
\]
Consequently, by Proposition \ref{mult_synth}, for $1 \leq p < \infty$,
\begin{align*}
\norm{R(c)}_{\ell^p_{L^p_v}}^p
&\leq \sum_i \norm{\sum_{k \in \Lambda^r_i} c^i_k \psi^i_k \partfi}_{L^p_v}^p
\\
&\lesssim M^p \sum_i \norm{c^i}_{\ell^p_v}^p
\\
&= M^p \norm{c}_{\ell^p_V}^p.
\end{align*}
So, $\norm{R(c)}_{\ell^p_{L^p_v}} \lesssim \norm{c}_{\ell^p_V}$. A similar computation shows that the same estimate is valid for $p=\infty$.
\end{proof}
\section{Quality statements about dual atoms}
\label{sec_quality}
The theory of localized frames asserts that, if the atoms in a frame are
sufficiently localized (in an abstract sense), then the dual atoms are also
localized. Theorem \ref{quality_theorem} below shows that if a family of
frames is sufficiently and uniformly localized, then the respective dual
families are also uniformly localized. This is relevant to the construction
in Section \ref{sec_surgery}.

To obtain this qualitative statement we need to know not only that the polynomial off-diagonal decay of a matrix $M$ is inherited by its pseudo-inverse $M^\dagger$, but also what qualities of the original matrix $M$ determine the constants governing the off-diagonal decay of $M^\dagger$. This extra information is essentially provided by the main result in \cite{su07-5}. Lemma \ref{pseudo_inverse} below states the required qualitative statement and Theorem \ref{quality_theorem} is just an application of it to our setting.
\begin{lemma}
\label{pseudo_inverse}
Let $\Gamma \subseteq \Rdst$ be a relatively separated set and
let $M \in B(\ell^2(\Gamma))$ be a positive operator.
Assume the following.
\begin{itemize}
 \item $M$ satisfies,
 \begin{align*}
 \abs{M_{k,j}} \leq C (1+\abs{k-j})^{-s}
 \qquad (k,j \in \Gamma),
 \end{align*}
for some constants $C>0$ and $s>d$.
 \item  The spectrum of $M$, satisfies,
 \begin{align*}
 \left(\sigma(M) \setminus \sett{0}\right) \cap B_A(0) = \emptyset,
 \end{align*}
for some $A>0$ (here $B_A(0) \subseteq \bC$ is the ball of radius $A$ centered at 0.)
 \item $\rel(\Gamma) \leq R$, for some $0 \leq R < \infty$.
\end{itemize}
Then $M^\dagger$, the Moorse-Penrose pseudo-inverse of $M$, satisfies,
\begin{align*}
 \abs{M^\dagger_{k,j}} \leq C' (1+\abs{k-j})^{-s}
 \qquad (k,j \in \Gamma),
\end{align*}
for some constant $C'$ that only depends on $C, s, d, A$ and $R$.
\end{lemma}
\begin{proof}
If we modify the hypothesis of the theorem so that $M$ is assumed to be an invertible normal operator
instead of a positive pseudo-invertible one (and hence $M^\dagger=M^{-1}$), then the conclusion follows from \cite[Theorem 4.1]{su07-5}.

The case of the pseudo-inverse is treated in \cite{su07-5}, but no explicit reference to the qualities involved in the off-diagonal decay of the pseudo-inverse is made. However, the proof given in \cite[Theorem 5.1]{su07-5} (see also \cite{fogr05}) can be slightly adapted to obtain a quantitative conclusion. We only sketch the modifications.

Under the assumptions of the theorem,
\begin{align}
\label{pseudo_inverse_formula}
M^\dagger = \frac{1}{2\pi i} \int_\gamma \frac{1}{z} \left( zI-M \right)^{-1} dz,
\end{align}
where the curve $\gamma$ is the rectangle with vertices  $A/2 \pm i$,
$\norm{M}+A/2 \pm i$ oriented anti-clockwise; here $\norm{M}$ denotes the
norm of $M$ in $B(\ell^2(\Gamma))$.
Consequently, for $k,j \in \Gamma$,
\begin{align}
\label{pseudo_inverse_formula_kj}
M^\dagger_{kj} = \frac{1}{2\pi i} \int_\gamma \frac{1}{z} \left( zI-M \right)^{-1}_{kj} dz.
\end{align}
Observe that $\norm{M}$ can be bounded in terms of $d, s, C$ and $R$ (by interpolating its $\ell^1 \to \ell^1$ and $\ell^\infty \to \ell^\infty$ norm) and that the length of $\gamma$ is $2\norm{M} + 2$. For $z$ in the curve $\gamma$, $\abs{z} \lesssim \norm{M}+1$ and $\abs{1/z} \leq 2/A$. Hence, it suffices to bound
the off-diagonal decay of the resolvent $\left(zI-M \right)^{-1}$ in terms of the allowed parameters.

Let $z$ lie in the curve $\gamma$. The distance from $z$ to $\sigma(M)$ is at least $m:=\min \sett{1,A/2}$, so
$\left(\sigma(zI-M) \setminus \sett{0} \right) \cap B_m(0) = \emptyset$. 
Moreover, for $k,j \in \Gamma$,
\begin{align*}
\abs{\left( zI-M \right)_{kj}}
&\leq \abs{z} \delta_{kj} + C(1+\abs{k-j})^{-s}
\\
&\lesssim (C+\norm{M}+1)(1+\abs{k-j})^{-s}.
\end{align*}
By \cite[Theorem 4.1]{su07-5}, the off-diagonal decay of $\left(zI-M \right)^{-1}$ 
is bounded by a constant depending only on allowed parameters.
\end{proof}
Now we apply this estimate to spline-type spaces.
\begin{theorem}
\label{quality_theorem}
Let $\spline \equiv \splinet(F, \Lambda)$ be a spline-type space, where the atoms $F$
satisfy,
\begin{equation*}
\abs{f_k(x)} \leq C \left(1+\abs{x-k}\right)^{-s}
\qquad (x \in \Rdst, k\in\Lambda),
\end{equation*}
for some constant $C>0$ and $s>0$. Assume the following.
\begin{itemize}
\item For each $i \in I$, we have a family of measurable functions
$\sett{\varphi^i_k}_{k \in \Lambda_i}$ that satisfy the following uniform concentration condition around their nodes $\Lambda_i$:
\begin{equation}
\abs{\varphi^i_k(x)} \leq C' \left(1+\abs{x-k}\right)^{-s}
\qquad (x \in \Rdst, k\in\Lambda_i),
\end{equation}
for some constant $C'>0$ (independent of $i$.)

\item The set of nodes $\Lambda_i$ are \emph{uniformly relatively separated}. That is,
\[
\sup_{i \in I} \rel (\Lambda_i) < \infty.
\]

\item Each family $\sett{\varphi^i_k}_k$ satisfies the (exterior) frame
inequality\footnote{Note that the functions $\varphi^i_k$ need not to
belong to $\splinet$, cf. Section \ref{sec_frames}.},
\begin{equation}
A \norm{f}_2^2 \leq \sum_k \abs{\ip{f}{\varphi^i_k}}^2 \leq B \norm{f}_2^2,
\end{equation} 
for $f \in \splinet$ and constants $0<A\leq B < \infty$ that are independent of $i$.
\end{itemize}
Then, the respective families of canonical dual frame sequences
$\sett{\psi^i_k}_k \subseteq \splinet$ satisfy,
\begin{equation}
\abs{\psi^i_k(x)} \leq D \left(1+\abs{x-k}\right)^{-s}
\qquad (x \in \Rdst, k\in\Lambda_i),
\end{equation}
for some constant D, independent of i.
\end{theorem}
\begin{proof}
Let $G \equiv \sett{g_k}_k$ be the canonical dual frame of $F$. By the results in \cite{fogr05}, there exists a constant $C''>0$ such that
\begin{equation*}
\abs{g_k(x)} \leq C'' \left(1+\abs{x-k}\right)^{-s}
\qquad (x \in \Rdst, k\in\Lambda).
\end{equation*}
For each $i \in I$ and $k \in \Lambda_i$, let $\overline{\varphi}^i_k$ be
the orthogonal projection of $\varphi^i_k$ on $\splinet$. Each of the functions has
the expansion,
\begin{align*}
\overline{\varphi}^i_k = \sum_{j\in \Lambda} \ip{\varphi^i_k}{g_j} f_j.
\end{align*}
Consequently using Lemmas \ref{conv_bound} and \ref{conv_nodes},
\begin{align*}
\abs{\overline{\varphi}^i_k}
&\lesssim C C' C'' \sum_{j\in \Lambda} w_{-s}(k-j) w_{-s}(\cdot-j)
\\
&\lesssim C C' C'' \rel(\Lambda) w_{-s}(\cdot-k).
\end{align*}
Since the exterior frame condition in the hypothesis is also satisfied by the functions
$\sett{\overline{\varphi}^i_k}_k$, we can replace each $\varphi^i_k$ by
$\overline{\varphi}^i_k$ and assume without loss of generality that
$\varphi^i_k \in \splinet$.

For each $i \in I$, consider the Gram matrix $M^i$ given by,
\[
M^i_{kj} := \ip{\varphi^i_k}{\varphi^i_j}
\qquad (k,j \in \Lambda_i).
\]
By Lemma \ref{conv_bound}, it follows that
\[
\abs{M^i_{k,j}} \leq K (1+\abs{k-j})^{-s}
\qquad (k,j \in \Lambda_i),
\]
for some constant $K$ that depends on $s$ and $C'$. Moreover, since each $\sett{\varphi^i_k}_k$ is a frame with bounds A and B,
the spectrum of $M^i$ satisfies,
\[
\sigma(M^i) \subseteq \sett{0} \cup [A,B].
\]

By Lemma \ref{pseudo_inverse}, the pseudo-inverse of $M^i$ satisfies
\[
\abs{(M^i)^\dagger_{k,j}} \leq K' (1+\abs{k-j})^{-s}
\qquad (k,j \in \Lambda_i),
\]
for some constant $K'$ independent of i.

Each of the dual elements $\psi^i_k$ is given by,
\[
\psi^i_k = \sum_{j \in \Lambda_i} (M^i)^\dagger_{k,j} \varphi^i_j.
\]
Therefore,
\begin{align*}
\abs{\psi^i_k(x)} &\leq CK' \sum_{j \in \Lambda_i} w_{-s}(k-j) w_{-s}(j-x)
\end{align*}
Using Lemma \ref{conv_nodes} (c) with $\Gamma := \Lambda_i - \sett{x}$ and
$k' := k-x$, it follows that
\begin{align*}
\abs{\psi^i_k(x)} &\leq K'' \rel(\Gamma) w_{-s}(x-k) = K'' \rel(\Lambda_i) w_{-s}(x-k).
\end{align*}
For some constant that $K''$ independent of i. Since the sets of nodes are uniformly relatively separated, the conclusion follows.
\end{proof}
\section{Applications}
\label{sec_applications}
\subsection{Spline-type spaces}
We now combine the results of Sections \ref{new_frames} and \ref{sec_quality} in a concrete statement.
\begin{theorem}
\label{splines_glue}
Let $\spline = \spline(F, \Lambda)$ be a spline-type space. Assume the following.
\begin{itemize}
\item The atoms $F$ satisfy the polynomial concentration condition around their nodes,
\begin{equation}
\label{spline_frame_loc}
\abs{f_k(x)} \leq C \left(1+\abs{x-k}\right)^{-(s+\alpha)}
\qquad (x \in \Rdst, k\in\Lambda),
\end{equation}
for some constants $C>0$, $s>d$ and $\alpha \geq 0$.

\item We are given a family of exterior frames for $\splinet$,
$\sett{\varphi^i_k}_{k \in \Lambda_i}$, $i \in I$,
that satisfy the following uniform polynomial concentration condition around their nodes $\Lambda_i$,
\begin{equation}
\label{spline_frames_loc_2}
\abs{\varphi^i_k(x)} \leq C' \left(1+\abs{x-k}\right)^{-(s+\alpha)}
\qquad (x \in \Rdst, k\in\Lambda_i, i \in I),
\end{equation}
for some constant $C'>0$.

\item The exterior frames $\left( \sett{\varphi^i_k}_k \right)_{i\in I}$
share a uniform lower (and upper) bound. That is,
\begin{equation}
\label{spline_frames_unif}
A \norm{f}_2^2 \leq \sum_k \abs{\ip{f}{\varphi^i_k}}^2 \leq B \norm{f}_2^2
\qquad (f \in \splinet),
\end{equation}
holds for some constants $0<A\leq B<\infty$.
\footnote{Observe that Equation \eqref{spline_frames_loc_2} already implies the existence of a uniform upper bound B.}

\item The sets of nodes $\Lambda_i$ are \emph{uniformly relatively separated}
(cf. Equation \eqref{unif_rel_sep}.)

\item We have a measurable covering of $\Rdst$, $\partt \equiv \sett{\parti}_{i \in I}$ that is \emph{uniformly locally finite} (cf. Equation \eqref{loc_finite_partition}.)
\end{itemize}
Then, for all sufficiently large values of $r>0$,
\[
\sett{\varphi^i_k: i \in I, d(k,E_i) \leq r}
\]
is an exterior Banach frame for $\splinepv$.

More precisely, if we define the index set $\Gamma^r := \bigcup_{i \in I}
\Lambda^r_i \times \sett{i}$ and the weight $V(k,i) := v(k)$, then the
analysis map
\begin{align*}
\splinepv &\to \el^p_V(\Gamma^r)
\\
f &\mapsto \left(\ip{f}{\varphi^i_k}\right)_{(k,i)}
\end{align*}
is bounded and left-invertible.

Moreover, the value of r may be chosen uniformly for all $1 \leq p \leq \infty$
and every class of $w_\alpha$-moderated weights for which the respective constants (cf. Equation \eqref{moderated_weight}) are uniformly bounded.
\end{theorem}
\begin{proof}
Combine Theorems \ref{gluying_frames} and \ref{quality_theorem}.
\end{proof}
\subsection{Shift invariant spaces}
\label{sis_sec}
As a corollary of Theorem \ref{splines_glue} we describe a method to piece
together bases of lattice translates. First recall some notation and facts
for shift-invariant spaces (see for example \cite{rosh95}, \cite{dedero94-1}
and \cite{bo00-2}.) Given a lattice\footnote{By a lattice $\Lambda$ we mean a
set of the form $\Lambda=A \Zdst$, where $A \in \Rst^{n \times n}$ is an
invertible matrix. This is sometimes called a full-rank lattice.}
$\Lambda \subseteq \Rdst$ and $f,g \in
L^2(\Rdst)$, the bracket product is defined by,
\begin{align*}
[f,g]_\Lambda(x) :=
\sum_{\lambda^\perp \in \Lambda^\perp}
\hat{f}(x+\lambda^\perp) \overline{\hat{g}(x+\lambda^\perp)}
\qquad (x \in \Rdst).
\end{align*}
Here, $\hat{f}(w):=\int_\Rdst f(x) e^{-2\pi i xw} dx$ is the Fourier transform of $f$ and
\begin{align*}
\Lambda^\perp := \set{\lambda^\perp \in \Rdst}{\ip{\lambda}{\lambda^\perp} \in \Zst,
\mbox{ for all }\lambda \in \Lambda}
\end{align*}
is the orthogonal lattice of $\Lambda$. Since the bracket $[f,g]$ is
$\Lambda^\perp$ periodic, it can be considered as a function on the torus
$\Rdst / \Lambda^\perp$.

The $\Lambda$ translates of a finite set of functions $\sett{f_1, \ldots,
f_N}$ form a Riesz sequence in $L^2(\Rdst)$ if an only if the matrix of
functions $\hat{G} \equiv \left( \hat{G}_{n,m} \right)_{1 \leq k,j \leq N}$
given by
\begin{align*}
\hat{G}(x)_{n,m} := [f_n,f_m]_\Lambda (x)
\qquad (x \in \Rdst),
\end{align*}
is uniformly invertible in the sense that all its eigenvalues are bounded
away from 0 and $\infty$, uniformly on $x$ (up to sets of null measure.)

Combining the theory of shift-invariant spaces with Theorem \ref{splines_glue} we get the following.
\begin{theorem}
\label{glue_sis}
Let $\Lambda \subseteq \Rdst$ be a lattice
and let $\splinet = \splinet(F, \Lambda \times \sett{1, \ldots, N})$ be
a finitely-generated shift invariant space where the atoms are given by,
\[F \equiv \sett{f_n(\cdot-\lambda): 1 \leq n \leq N, \lambda \in \Lambda}.\]

Assume the following.
\begin{itemize}
 \item The atoms form a Riesz basis of $\splinet$ and satisfy the following
decay condition,
\begin{equation}
\abs{f_n(x)} \leq C (1+\abs{x})^{-(s+\alpha)}
\qquad (1 \leq n \leq N),
\end{equation}
for some constants $C>0$, $\alpha \geq 0$ and $s>d$.
 \item We have a measurable covering of $\Rdst$, $\partt \equiv \sett{\parti}_{i \in I}$ that is uniformly locally finite (cf. Equation \eqref{loc_finite_partition}.)
 \item We are given a family of measurable functions
\[\set{g_n^i: \Rdst \to \bC}{i\in i, 1\leq n \leq N}\] satisfying the decay condition,
\begin{equation}
 \abs{g_n^i(x)} \leq C' (1+\abs{x})^{-(s+\alpha)}
\qquad (1 \leq n \leq N),
\end{equation}
for some constant $C'>0$ (independent of $i$ and $n$.)
 \item The matrices of functions $\left( \hat{G}^i_{n,m} \right)_{1 \leq n,m \leq N}$
given by
\begin{align*}
 \hat{G^i}(x)_{n,m} := [f_n,g^i_m]_\Lambda (x)
\qquad (x \in \Rdst / \Lambda^\perp),
\end{align*}
are uniformly bounded and invertible in the sense that each $\hat{G^i}(x)$ is invertible and
\begin{align*}
\sup_{x,i} \norm{\hat{G^i}(x)}, \sup_{x,i} \norm{\hat{G^i}(x)^{-1}} < \infty.
\end{align*}
\end{itemize}
Then, for all sufficiently large values of $r>0$, the set
\begin{align*}
 \set{g^i_n(\cdot-\lambda)}{i \in I, 1 \leq n \leq N, \lambda \in \Lambda,
d(\lambda, \parti) \leq r},
\end{align*}
is a Banach frame for $\splinepv$, for all $1 \leq p \leq \infty$ and all
strictly $w_\alpha$-moderated weights $v$. More precisely, if we define the
index set
\begin{align*}
 \Gamma^r :=
\set{(i,n,\lambda) \in I \times \sett{1, \ldots, N} \times
\Lambda}{d(\lambda, E_i) \leq r}
\end{align*}
and the weight $V(i,n,\lambda):=v(\lambda)$ on it, then the analysis map
\begin{align*}
\splinepv &\to \el^p_V(\Gamma^r)
\\
f &\mapsto \left(\ip{f}{g^i_{n,\lambda}}\right)_{(i,n,\lambda)}
\end{align*}
is bounded and left-invertible.
\end{theorem}
\begin{rem}
The theorem is stated for bases just for simplicity. Using the tools from
\cite{rosh95}, \cite{dedero94-1} and \cite{bo00-2} it can be reformulated
for frames.
\end{rem}
\begin{proof}
Let $A$ and $B$ be the Riesz basis bounds of $F$. Also let
$A' := \sup_{x,i} \norm{\hat{G^i}(x)^{-1}}$ and
$B' := \sup_{x,i} \norm{\hat{G^i}(x)}$.
Using the fiberization theory in \cite{rosh95}, \cite{dedero94-1} and
\cite{bo00-2}, for each $x \in \Rdst / \Lambda^\perp$, the system
$\sett{(\hat{f_1}(x+k))_k, \ldots, (\hat{f_N}(x+l))_k}$ is a Riesz basis with
bounds $A$ and $B$ for some subspace $\splinet_x \subseteq
\ell^2(\Lambda^\perp)$. Since its cross-gramian matrix with the system
$\sett{(\hat{g^i_1}(x+k))_k, \ldots, (\hat{g^i_N}(x+k))_k}$ is invertible,
it follows that this latter system is a Riesz projection basis for the
subspace $\splinet_x$ with bounds $B^{-1}A'^{-2}$ and $(B')^2 A^{-1}$.
Invoking again the fiberization theory, it follows that the $\Lambda$
translates of $\sett{g^i_1, \ldots, g^i_N}$ are a projection basis for
$\splinet$ with bounds $\approx$ $B^{-1}A'^{-2}$ and $(B')^2 A^{-1}$
(the implicit constant depends on the volume of the lattice $\Lambda$.) We
can now apply Theorem \ref{splines_glue}.
\end{proof}
\begin{rem}
In \cite{bo00-2} no results for projection bases nor exterior
frames are explicitly given. However, it is proved there (and also in
\cite{dedero94-1})
that the orthogonal projector onto a shift-invariant space operates
fiberwise, so the desired extension follows. For further results on exterior
frames for shift-invariant spaces see \cite{chel04} and \cite{chel05}.
\end{rem}

\subsection{Sampling}
Applying Theorem \ref{splines_glue} to the reproducing kernels of a (smooth)
spline-type space we get the following.
\begin{theorem}
Let $\spline = \spline(F, \Lambda)$ be a spline-type space generated by 
a family of continuous atoms $F \subseteq C^0(\Rdst)$ that satisfy,
\begin{equation*}
\abs{f_k(x)} \leq C \left(1+\abs{x-k}\right)^{-(s+\alpha)},
\quad (x\in\Rdst, k \in \Lambda),
\end{equation*}
for some $s>d$, $C>0$ and $\alpha \geq 0$.

Assume the following.
\begin{itemize}
 \item $\partt \equiv \sett{\parti}_{i\in I}$ is a uniformly locally finite measurable covering of $\Rdst$ (cf. Equation \eqref{loc_finite_partition}).
 \item For each $i \in I$, we have a set $X_i \subseteq \Rdst$ and this collection of sets
is uniformly relatively separated (i.e $\sup_i \rel(X_i) < \infty$.)
 \item For each of the sets $X_i$, the following sampling inequality
\begin{equation}
\label{sampling_i}
A \norm{f}_2^2 \leq \sum_{x \in X_i} \abs{f(x)}^2 \leq B \norm{f}_2^2,
\end{equation}
holds for all $f \in \splinet$ and some constants $0<A\leq B < \infty$
independent of i.
\end{itemize}
For each $r>0$, let
\begin{equation*}
X^r := \sett{(i,x): i \in I, x \in X_i, d(x,\parti) \leq r}.
\end{equation*}
Then, for all sufficiently large $r>0$, there exists constants $0<A^r \leq B^r < \infty$
such that the sampling inequality,
\begin{equation}
\label{sampling_pv}
A^r \norm{f}_{L^p_v} \leq
\left( \sum_{(i,x) \in X^r} \abs{f(x)}^p v(x)^p \right)^{1/p} \leq B^r \norm{f}_{L^p_v},
\end{equation}
holds for all $1 \leq p \leq \infty$ (with the usual adjustment for $p=\infty$), all strictly $w_\alpha$-moderated weights $v$, and all $f \in \splinepv$.
\end{theorem}
\begin{rem}
For any class of $w_\alpha$-moderated weights for which the respective constants (cf. Equation \eqref{moderated_weight}) are uniformly bounded, the conclusion of the theorem still holds.
\end{rem}

\begin{proof}
First observe that since $F \subseteq C^0(\Rdst)$, Theorem \ref{spline_norm_equiv} applies with $\Bsp=C^0$ and consequently $\splinepv \subseteq C^0$. The norm equivalence of Theorem \ref{spline_norm_equiv} also implies that $\splinet$ is a reproducing-kernel Hilbert space. We already know that $F$ has a dual frame $G \equiv \sett{g_k}_k$ satisfying a polynomial decay condition,
\begin{align*}
\abs{g_k(x)} \leq C' \left(1+\abs{x-k}\right)^{-(s+\alpha)},
\end{align*}
for some constant $C'>0$. The functional $f \mapsto f(x_0)$ is represented by the function $K_{x_0} \in \splinet$ given by
\begin{equation}
\label{kx0}
K_{x_0} = \sum_{k \in \Lambda} \overline{g_k}(x_0) f_k.
\end{equation}
We will apply Theorem \ref{splines_glue} to the family of frames,
\begin{align*}
\sett{K_x}_{x \in X_i}
\qquad (i \in I).
\end{align*}
To this end, observe that Equation \eqref{sampling_i} implies that this family satisfies the condition on Equation \eqref{spline_frames_unif} of Theorem \ref{splines_glue}. We only need to check the condition on Equation \eqref{spline_frames_loc_2} for the family of reproducing kernels.

For $x \in X_i$, using Equation \eqref{kx0}, we estimate,
\begin{align*}
\abs{K_x} &\leq CC' \sum_{k \in \Lambda} w_{-(s+\alpha)}(x-k) w_{-(s+\alpha)}(\cdot-k).
\end{align*}
Using Lemma \ref{conv_nodes} (c) with $\Gamma := \Lambda - \sett{x}$, it follows that
\begin{align*}
\abs{K_x} &\leq CC'\rel(\Gamma) w_{-(s+\alpha)}(\cdot-x) = K'' \rel(\Lambda) w_{-(s+\alpha)}(\cdot-x).
\end{align*}
Now we can apply Theorem \ref{splines_glue} to obtain the desired conclusion.
\end{proof}

\subsection{Gabor molecules}
Let $\phi:\Rdst \to \Rst$, $\phi(x):=\pi^{-d/4} e^{-\frac{\abs{x}^2}{2}}$ be
the Gaussian normalized in $L^2$. The \emph{Short-Time Fourier Transform}
with respect to $\phi$ of a test function $f \in \SchRd$ is defined by
\begin{equation}
\label{def_stft}
\stft_\phi f(x,w) := \ip{f}{M_w T_x\phi}.
\end{equation}
Here, $T_x$ is the \emph{translation operator} given by
\begin{align*}
T_x(f)(y):=f(y-x),
\end{align*}
and $M_w$ is the \emph{modulation operator} given by
\begin{align*}
M_w(f)(y):=e^{2\pi i w y}f(y).
\end{align*}
The definition in Equation \eqref{def_stft} extends to tempered
distributions. The time-frequency shift $\pi(x,w)$ is defined by $\pi(x,w) :=
M_w T_x$.

For $1 \leq p \leq \infty$ and a weight $v$, the \emph{modulation space}
$M^p_v$ is defined as
\begin{equation*}
 M^p_v := \set{f \in \SchpRd}{\stft_\phi f \in L^p_v(\Rtdst)},
\end{equation*}
and given the norm $\norm{f}_{M^p_v} := \norm{\stft_\phi f}_{L^p_v}$, which
makes it a Banach space
(see \cite[Chapter 11]{gr01}.) $M^0_v$ is similarly defined using
$C^0_v$ instead of $L^\infty_v$.

For an adequate lattice, $\Lambda \subseteq \Rdst \times \Rdst$, the Gabor system
$\sett{M_w T_x \phi: (x,w) \in \Lambda}$ is a frame for $L^2(\Rdst)$.
Consider the family of functions $F \equiv \sett{f_k}_{k\in\Lambda} \subseteq
L^2(\Rdst \times \Rdst)$ defined by $f_k := \stft_\phi (M_w T_x \phi)$, where
$k=(x,w)$. Since $\stft_\varphi: L^2(\Rdst) \to L^2(\Rdst \times \Rdst)$ is
an isometry, it follows that $F$ forms a frame sequence in $L^2(\Rdst \times
\Rdst)$.

Since $\stft_\phi \phi \in \SchRd$ (see \cite[Theorem 11.2.5]{gr01}), for any $s>0$ there exists a constant $C_s>0$ such that
\[
\abs{\stft_\phi \phi(z)} \leq C_s \left( 1 + \abs{z} \right)^{-s}.
\]
Since $\abs{f_k} = \abs{\stft_\phi \phi(\cdot-k)}$ (see \cite[Equation 3.14]{gr01}) it follows that,
\begin{equation}
\label{loc_stft_gauss}
\abs{f_k(z)} \leq C_s \left( 1 + \abs{z-k} \right)^{-s}
\qquad (z \in \Rtdst, k \in \Lambda).
\end{equation}
Consequently, by Example \ref{poly_amalgam}, we know that $\spline = \spline(F, \Lambda)$ is a spline-type space.

Observe that for polynomially moderated weights v and $1 \leq p < \infty$,
$\stft_\phi$ maps, by definition, the modulation space $M^p_v$ isometrically
onto $\splinepv$. For $p=\infty$, the same statement is true replacing
$M^\infty_v$ for $M^0_v$.

In view of this, Theorem \ref{gluying_frames} can be reformulated for Gabor molecules.
\begin{theorem}
\label{gluying_gabor_molecules}
\mbox{}

\begin{itemize}
\item Let $\partt \equiv \sett{\parti}_{i\in I}$ be a uniformly locally finite measurable covering of $\Rdst \times \Rdst$ (cf. Equation \eqref{loc_finite_partition}.)
\item For each $i \in I$, let $G^i \equiv \sett{g^i_k}_{k\in\Lambda_i}$ be a frame for $L^2(\Rdst)$ with lower bound $A_i$ and suppose that $A:=\inf_i A_i >0$.
\item Suppose that the sets of time-frequency nodes $\Lambda_i \subseteq \Rdst \times \Rdst$ are uniformly relatively separated (i.e $\sup_{i \in I} \rel (\Lambda_i) < \infty$.)
\item Assume that the molecules $G^i$ satisfy the following uniform time-frequency concentration condition,
\begin{equation}
\label{tf_loc}
\abs{\stft_\phi g^i_k(z)} \leq C \left(1+\abs{z-k}\right)^{-(s+\alpha)}
\qquad (z \in \Rdst \times \Rdst, k \in\Lambda_i),
\end{equation}
for some constants $C>0$, $s>2d$ and $\alpha \geq 0$ (independent of i.)
\end{itemize}
Then, for all sufficiently large $r>0$, the system
\[
\sett{g^i_k: i \in I, k \in \Lambda_i, d(k,\parti) \leq r}
\]
is a Banach frame simultaneously for all the modulation spaces $M^p_v$,
for all strictly $w_\alpha$-moderated weights v and $1 \leq p < \infty$.
The same is true for $p=\infty$, replacing $M^\infty_v$ for $M^0_v$.

More precisely, if we set,
\begin{equation*}
\Gamma^r := \sett{(i,k): i \in I, k \in \Lambda_i, d(k,\parti) \leq r},
\end{equation*}
and define a weight $V$ on $\Gamma^r$ by
\[
V(k,i):=v(k), 
\]
then, the coefficients map given by
\begin{align*}
M^p_v &\to \el^p_V(\Gamma^r)
\\
f &\mapsto \left( \ip{f}{g^i_k} \right)_{(i,k)}
\end{align*}
is bounded and left-invertible, for all sufficiently large values of $r$.
\end{theorem}
\begin{rem}
For any class of $w_\alpha$-moderated weights for which the respective constants (cf. Equation \eqref{moderated_weight}) are uniformly bounded, it is also possible to choose a value of $r>0$ for which the conclusion of the theorem holds.
\end{rem}
\begin{proof}
Consider the spline-type space $\splinet = \stft_\phi(L^2(\Rdst))$ from the
discussion above. Define the functions,
\begin{align*}
\varphi^i_k := \stft_\phi(g^i_k)
\qquad (i \in I, k \in \Lambda_i).
\end{align*}
Since $\stft_\varphi: L^2(\Rdst) \to L^2(\Rdst \times \Rdst)$ is an isometry, each of the families $\sett{\varphi^i_k}_k$ is a frame for $\splinet$ with lower bound $A$. Moreover, Equation \eqref{tf_loc} implies that these families share a uniform polynomial concentration condition. This condition is also shared by the atoms $\sett{f_k}_k$ because Equation \eqref{loc_stft_gauss} holds for any value of $s>0$. The theorem now follows from Theorem \ref{splines_glue} and the fact that
$:\stft_\phi:M^p_v \to \splinepv$ is a surjective isometry (with the discussed modification for $p=\infty$.)
\end{proof}
\begin{rem}
\label{smoothness_stft}
Observe that since we have identified the range of the STFT (with a fixed window) with a spline-type space, we get from Theorem \ref{spline_norm_equiv} that, on the range of the STFT, the $L^p_v$ and $W(L^\infty, L^p_v)$ norms are equivalent (the class of weights $v$ for which this is true depends on the time-frequency localization of the window function; in the case of the Gaussian window, any polynomial weight $w_\alpha$ with $\alpha \geq 0$ will work.) Results of this kind can be found in Chapter 12 of \cite{gr01}, see for example Proposition 12.1.11 there.
\end{rem}
\begin{rem}
Finally observe that the argument given can be used to combine not only time-frequency concentrated frames for $L^2(\Rdst)$ but also frames for \emph{proper subspaces} $S \subseteq L^2(\Rdst)$. Simply let $\splinet=\stft_\phi(S)$ and apply the same argument as above.
\end{rem}
For completeness, we give a version of Theorem \ref{gluying_gabor_molecules} for pure time-frequency atoms. This gives general sufficient conditions for the existence of the so called \emph{quilted Gabor frames}, recently introduced in \cite{dofe09}.
\begin{corollary}
\mbox{}

\begin{itemize}
\item Let $\partt \equiv \sett{\parti}_{i\in I}$ be a uniformly locally finite measurable covering of $\Rdst \times \Rdst$ (cf. Equation \eqref{loc_finite_partition}.)
\item For each $i \in I$, let $G^i \equiv \sett{T_j M_k g^i: (k,j) \in \Lambda_i}$ be a Gabor frame for $L^2(\Rdst)$ with lower bound $A_i$ and suppose that $A:=\inf_i A_i >0$.
\item Suppose that the sets of time-frequency nodes $\Lambda_i \subseteq \Rdst \times \Rdst$ are uniformly relatively separated.
\item Assume that the windows $\sett{g^i}_i$ satisfy the following uniform time-frequency concentration condition,
\begin{equation*}
C := \sup_i \norm{g^i}_{M^\infty_{w_{s+\alpha}}} < + \infty,
\end{equation*}
for some constants $s>2d$ and $\alpha \geq 0$ (independent of i).
\end{itemize}
Then, for all sufficiently large $r>0$, the system
\[
\sett{T_j M_k g^i: i \in I, (k,j) \in \Lambda_i, d((k,j),\parti) \leq r}
\]
is a Banach frame simultaneously for all the modulation spaces $M^p_v$,
for all strictly $w_\alpha$-moderated weights v and $1 \leq p < \infty$.
The same is true for $p=\infty$, replacing $M^\infty_v$ for $M^0_v$.
\end{corollary}
\begin{proof}
Observe that,
\begin{align}
\abs{\stft_\phi(T_j M_k g^i)} = \abs{\stft_\phi(g^i)(\cdot-\lambda)}
\leq C w_{-(s+\alpha)} (z-\lambda),
\end{align}
where $\lambda:=(j,k) \in \Lambda_i$. Therefore, we can apply Theorem \ref{gluying_gabor_molecules}.
\end{proof}
\subsection{Gabor multipliers}
Now we give an application of the frame surgery scheme to Gabor multipliers. We follow largely the approach in \cite{fe02}. For a general background on Gabor multipliers see \cite[Chapter 5]{fest03}.

Given a lattice in the time-frequency plane $\Lambda \subseteq \Rtdst$ and
two families of functions $F \equiv \sett{f_1, \ldots, f_N}, G \equiv
\sett{g_1, \ldots, g_N} \subseteq L^2(\Rdst)$ we consider the class of
operators,
\begin{align*}
\gabt_{F,G} :=
\set{
\sum_{n=1}^N \sum_{\lambda \in \Lambda}
m_n(\lambda) \ip{\dash}{\pi(\lambda)g_n} \pi(\lambda)f_n
}
{
m_n \in \ell^2(\Lambda)
},
\end{align*}
where $\pi(\lambda)$ is the time-frequency shift $\pi(\lambda) := M_w T_x$, if $\lambda=(x,w)$. The convergence of the series defining the class $\gabt$ requires additional assumptions (see below.) The operators in this class are called the \emph{Gabor multipliers} associated with the time-frequency atoms $(F,G)$ and the lattice $\Lambda$.

For $f,g \in L^2(\Rdst)$ we use the notation $P_{f,g} := \ip{\dash}{g}f$ for the corresponding rank-one operator. Furthermore, for a point $(x,w) \in \Rdst \times \Rdst$ we let the time-frequency shifts act on an operator $T$ by
\begin{align*}
 \rho(x,w) (T) := M_w T_x T T_{-x} M_{-w} = \pi(x,w) T \pi(x,w)^*.
\end{align*}

Every linear operator $T$ mapping continuously $\SchRd$ into $\SchpRd$ admits
a distributional kernel $K(T) \in \SchpRtd$. The \emph{Kohn-Nirenberg symbol}
of $T$ is defined in terms of $K$ by
\begin{align*}
\sigma(T)(x,w) := \int_\Rdst K(T)(x,x-s) e^{-2 \pi i sw} ds.
\end{align*}
From this definition it follows that the Kohn-Nirenberg map defines and
isometry between the class of Hilbert-Schmidt operators and $L^2(\Rtdst)$.
The important property for us is that the Kohn-Nirenberg map interwines the
action $\rho$ with the regular action of $\Rdst \times \Rdst$ (by
translations.) That is, 
\begin{align*}
\sigma(\rho(z)T) = \sigma(T)(\cdot -z)
\qquad ( z \in \Rdst \times \Rdst). 
\end{align*}
We see then that the Kohn Nirenberg map $KN: T \mapsto \sigma(T)$ relates
the class $\gabt$ to a shift-invariant space $\splinet(F,G) :=
KN(\gabt_{F,G})$ given by,
\begin{align*}
\splinet = \set{
\sum_{n=1}^N \sum_{\lambda \in \Lambda}
m_n(\lambda) \sigma(P_{f_n, g_n})(\cdot - \lambda)
}
{
(m_n) \in \ell^2(\Lambda)
}.
\end{align*}
The Kohn-Nirenberg symbol of the projector $P_{f,g}$ is explicitly given by,
\begin{align}
\label{KN_P}
\sigma(P_{f,g})(x,w) = f(x) \overline{\hat{g}}(w) e^{-2 \pi i xw},
\end{align}
so its 2d Fourier transform is
\begin{align*}
\widehat{\sigma(P_{f,g})}(x,w) = \stft_g f (-w,x).
\end{align*}
Consequently, the inner product between the building blocks of $\splinet$
is given by,
\begin{align*}
\ip{\sigma(P_{f_n, g_n})}{\sigma(P_{f_m, g_m})} = \ip{\stft_{g_n} f_n}{\stft_{g_m} f_m},
\end{align*}
whereas, with the notation $z^*=(-w,x)$ for $z=(x,w)$,
their bracket product (see Section \ref{sis_sec}) is given by
\begin{align}
\label{gabbr}
[\sigma(P_{f_n, g_n}),\sigma(P_{f_m, g_m})]_{\Lambda} (z)
= \sum_{\lambda^\perp \in \Lambda^\perp}
\stft_{g_n} f_n(z^*-\lambda^\perp)
\overline{\stft_{g_m} f_m(z^*-\lambda^\perp)}.
\end{align}
Hence, the theory of shift-invariant spaces (see \cite{rosh95}, \cite{dedero94-1} and \cite{bo00-2}) implies the following.
\begin{prop}
\label{gabmult_prop}
The set $\set{\ip{\dash}{\pi(\lambda)g_n} \pi(\lambda)f_n}{\lambda \in \Lambda}$
is a Riesz sequence in the space of Hilbert-Schmidt operators if and only if
the matrix of functions $\hat{G} = \hat{G}(F,G) \equiv (\hat{G}_{n,m})_{1 \leq n,m \leq N}$, given by,
\begin{align}
\label{G_fg}
 \hat{G}_{n,m}(z) = \sum_{\lambda^\perp \in \Lambda^\perp}
\stft_{g_n} f_n(z-\lambda^\perp)
\overline{\stft_{g_m} f_m(z-\lambda^\perp)},
\end{align}
is uniformly bounded and invertible (that is, its eigenvalues are bounded away from 0 and $\infty$, uniformly on $z$.)
\end{prop}
\begin{rem}
Observe that, for time-frequency concentrated windows,
since by Remark \ref{smoothness_stft} the STFT of an $L^2$ function belongs
to the amalgam space $W(C_0,L^2)$, it follows that the entries of the matrix
in Equation \eqref{G_fg} are continuous periodic functions. Therefore, that
matrix will be uniformly invertible if it is invertible at every point.
\end{rem}
\begin{proof}
The only observation to complete the proof is that, since the condition in Equation \eqref{G_fg} is required for every $z \in \Rtdst$, we can drop the change of coordinates $z \mapsto z^*$ in the bracket product.
\end{proof}
Consequently, in the situation of Proposition \ref{gabmult_prop}, any
operator $T \in \gabt(F,G)$ can be stably recovered from its \emph{lower
symbol}
\begin{align*}
\left( 
\ip{T}{P_{\pi(\lambda)f_n, \pi(\lambda)g_n}}_{HS}
: \lambda \in \Lambda
\right),
\end{align*}
where $\ip{\cdot}{\cdot}_{HS}$ denotes the Hilbert-Schmidt inner product.
We can now reformulate Theorem \ref{glue_sis} in this context.
\begin{theorem}
\label{glue_gab_mult}
Let a lattice $\Lambda \subseteq \Rtdst$ and a uniformly locally finite
measurable covering of the time-frequency plane $\partt \equiv
\sett{\parti}_{i \in I}$ be given.

Let $f_1, \ldots, f_N, g_1, \ldots g_N \in L^2(\Rdst)$ be such that the matrix
$\hat{G}(F,G)$ on Equation \eqref{G_fg} is uniformly invertible and suppose
that these atoms satisfy,
\begin{align}
\label{gm_decay}
 \abs{f_n(x)} &\leq C (1+\abs{x})^{-s},
\\
 \abs{\hat{g}_n(w)} &\leq C (1+\abs{w})^{-s},
\end{align}
for some constants $C>0$ and $s>d$.

Let families $\sett{f^i_1, \ldots, f^i_N}, \sett{g^i_1, \ldots g^i_N} \subseteq L^2(\Rdst)$, $i \in I$
be given. Assume the following.
\begin{itemize}
 \item The given families satisfy,
\begin{align}
\label{gm_decay_2}
 \abs{f^i_n(x)} &\leq C' (1+\abs{x})^{-s},
\\
 \abs{\hat{g}^i_n(w)} &\leq C' (1+\abs{w})^{-s},
\end{align}
for some constant $C'>0$ (independent of $i$ and $n$.)
 \item The matrices of functions $\left( \hat{G}^i_{n,m} \right)_{1 \leq n,m \leq N}$
given by
\begin{align*}
 \hat{G^i}(z)_{n,m} :=
\sum_{\lambda^\perp \in \Lambda^\perp}
\stft_{g_n} f_n(z-\lambda^\perp)
\overline{\stft_{g^i_m} f^i_m(z-\lambda^\perp)}
\qquad (z \in \Rdst \times \Rdst).
\end{align*}
are uniformly bounded and invertible in the sense that each $\hat{G^i}(z)$ is invertible and
\begin{align*}
\sup_{z,i} \norm{\hat{G^i}(z)}, \sup_{z,i} \norm{\hat{G^i}(z)^{-1}} < \infty.
\end{align*}
\end{itemize}
Then, for all sufficiently large values of $r>0$, any Gabor multiplier
$T \in \gabt(F,G)$ can be stably recovered in Hilbert-Schmidt norm from its \emph{mixed lower symbol}
\begin{align}
\label{mixed_lower_symbol}
\left(
\ip{T}{P_{\pi(\lambda)f^i_n, \pi(\lambda)g^i_n}}_{HS}
: i \in I, 1\leq n \leq N,\lambda \in \Lambda, d(\lambda, \parti) \leq r
\right).
\end{align}
\end{theorem}
\begin{rem}
As we have seen, the theorem also establishes an uniform equivalence
between the $\ell^p_v$ norm of the coefficients in Equation \eqref{mixed_lower_symbol}
and the $L^p_v$ norm of the Kohn-Nirenberg symbol of $T$, for $1 \leq p \leq \infty$ and a certain class of weights.
\end{rem}
\begin{proof}
By the discussion above, in order to apply Theorem \ref{glue_sis} we need to
observe that the Kohn-Nirenberg symbols of all the atoms are adequately localized. 
This follows from Equation \eqref{KN_P} and the fact that
$w_{-s}(x) w_{-s}(w) \leq w_{-2s}(x,w)$, for $x,w \in \Rdst$.
\end{proof}
\section{Acknowledgement}
The author is greatly indebted to Prof. Hans Feichtinger, who motivated and
encouraged this research and provided his perspective on several key topics.

This paper was written during a long-term visit to NuHAG in which the author
was partially supported by the EUCETIFA Marie Curie Excellence Grant
(FP6-517154, 2005-2009). The author holds a fellowship from the CONICET and
thanks this institution for its support. His research is also partially
supported by grants: PICT06-00177, CONICET PIP 112-200801-00398 and UBACyT
X149.
\bibliographystyle{abbrv}

\end{document}